\documentclass[12pt,a4paper]{amsart}
\usepackage{amsmath, amsfonts, amssymb,amsthm}
\usepackage[all]{xy}

\newtheorem{thm}{Theorem}
\newtheorem{cor}{Corollary}

\newtheorem{lem}{Lemma}
\newtheorem{conj}{Conjecture}
\newtheorem{defn}{Definition}
\newtheorem{prop}{Proposition}

\numberwithin{equation}{section} \numberwithin{thm}{section}
\numberwithin{lem}{section}
\numberwithin{problem}{section}
\numberwithin{prop}{section}
\numberwithin{cor}{section}
\numberwithin{conj}{section}
\usepackage{bbding}

\newcommand{\Z}{\mathbb Z}
\newcommand{\E}{\mathbb E}

\newcommand{\Po}{\mathbb{P}}

\renewcommand{\le}{\leqslant}
\renewcommand{\ge}{\geqslant}

\def\mod{\mathop{\rm mod}\nolimits}

\begin{document}

\title{Sidon basis}
\author{Javier Cilleruelo}
\address{Instituto de Ciencias Matem\'aticas (ICMAT) and Departamento de Matem\'aticas\\
Unversidad Aut\'onoma de Madrid\\
Madrid 28049\\
Espa\~na.
}
\email{franciscojavier.cilleruelo@uam.es}

\date{\today}


\begin{abstract} Erd\H{o}s conjectured the existence of an infinite Sidon sequence of positive integers which is also an asymptotic basis of order $3$.
We make progress towards this conjecture in several directions. First we prove the conjecture for all cyclic groups $\Z_N$ with $N$ large enough.

In second place we prove by probabilistic methods that there is an infinite  $B_2[2]$ sequence which is an asymptotic basis of order $3$.

Finally we prove that for all $\varepsilon >0$ there is a Sidon sequence which is an asymptotic basis of order $3+\varepsilon $,
that is to say, any positive sufficiently large integer $n$ can be written as a sum of $4$ elements of the sequence, one of them smaller than $n^{\varepsilon}$.
\end{abstract}

\maketitle
\section{Introduction}
 A sequence of positive integers $A$ is a Sidon basis of order $h$ if all the sums $a+a',\ a\le a',\ a,a'\in A$ are distinct (Sidon property) and
 if  any positive integer $n$, large enough, can be written as a sum of $h$ elements of $A$ (asymptotic basis of order $h$).
 It is not difficult to prove that there cannot be a Sidon basis of order $2$, however Erd\H os stated the following (\cite{E2}, \cite{ESS}, \cite{ESS2}):
 \begin{conj}\label{H}
There is a Sidon  basis of order $3$.
\end{conj}
While we are not able to prove this we prove some results approaching it. The first one is the modular version of Conjecture \ref{H}.
\begin{thm}\label{modular2}For  all $N$ large enough, the cyclic group $\Z_N$ contains a Sidon set $S\subset \Z_N$ which is  a basis of order $3$ in $\Z_N$.
\end{thm}

  We will proof this theorem in section \S 2.
An important ingredient in the proof of Theorem \ref{H} is a result  of Granville, Shparlinski and Zaharescu \cite{G} about distributions in the s-dimensional torus of points coming from  curves in $\mathbb F_p^r$. We start section \S 2 with a weaker theorem that has a cleaner proof and  is  used in the proof of Theorem \ref{Main2}.

Theorem \ref{Main2} is concerned with $B_2[g]$ sequences, a natural generalization of Sidon sequences.
\begin{defn}
A sequence of positive integers $A$ is a $B_2[g]$ sequence if any integer $n$ has at most $g$ representations of the form $n=a+a',\ a\le a',\  a,a'\in A$.
The $B_2[1]$ sequences are just the Sidon sequences.
\end{defn}

Erd\H os claimed in \cite{E2} that there exists a $B_2[g]$ sequence of positive integers
which is an asymptotic basis of order $3$ for some $g$ and he asked for the minimum possible $g$ (see also \cite{K}).
 While Conjecture \ref{H}  would imply that the minimum is $g=1$, we prove that $g\le 2$.
\begin{thm}\label{Main2}
There exists a $B_2[2]$ sequence of positive integers which is an asymptotic basis of order $3$.
\end{thm}
We next introduce a new generalization of basis that appears in the statement of our strongest approximation to Conjecture \ref{H}.
\begin{defn}
For any $\varepsilon,\ 0<\varepsilon< 1$ we say that $A$ is an asymptotic basis of order $h+\varepsilon$
if any sufficiently large positive integer $n$ can be written as a sum of $h+1$ elements of $A$, one of them smaller than $n^{\varepsilon}$:
$$n=a_1+\cdots +a_{h+1},\quad a_1,\dots,a_{h+1}\in A,\quad a_{h+1}\le n^{\varepsilon}.$$
We say that $A$ is a Sidon basis of order $h+\varepsilon$ if in addition it is a Sidon sequence.
\end{defn}
\begin{thm}\label{Main1}
For any $\varepsilon>0$ there exists a Sidon basis of order $3+\varepsilon$. In other words: for any $\varepsilon>0$ there exists a Sidon sequence $A$ of positive integers such that all positive integer $n$, large enough, can be written as
$$n=a_1+a_2+a_3+a_4,\quad a_1,a_2,a_3,a_4\in A,\quad a_4\le n^{\varepsilon}.$$
\end{thm}
We mention a couple of previous related results. Deshoulliers and Plagne \cite{DP} constructed a Sidon basis of order 7 and  Kiss \cite{K2} proved the existence of a Sidon basis of order to $5$. See also note 1.1.

To prove Theorems \ref{Main2} and \ref{Main1} we will use the probabilistic method invented by Erd\H os and Renyi \cite{ER}.
 In the study of sequences satisfyng certain additive properties they considered a probabilistic space $\mathcal S(\gamma)$ of sequences of positive integers where all the events $x\in A$ are independent and $\Po(x\in A)=x^{-\gamma}$.

On the one hand an easy application of this method shows that
if $\gamma>3/4$ then almost all sequences in $\mathcal S(\gamma)$ are Sidon sequences (after we remove a finite number of elements  from the sequence).

On the other hand  Erd\H os and Tetali \cite{ET} proved that if $\gamma<1-1/h$, then almost all sequences are asymptotic basis of order $h$.
Thus, for any $\gamma$ in the interval $(3/4,4/5)$ we have that almost all sequences in $\mathcal S(\gamma)$ are simultaneously Sidon sequences and asymptotic basis of order $5$. This is the argument used in \cite{K2}.

In order to get basis of order $3+\varepsilon$ we need to take $\gamma$ close  to $2/3$.
In this case the sequences are far from being Sidon sequences since we expect  infinitely many repeated sums.
A way to circumvent this obstacle is  to remove the elements involved in such repetitions to obtain a true Sidon sequence. This general idea, which is called alteration method or deletion technique is standard in the probabilistic method (see \cite{AS}) and has been used previously in a similar context (\cite{C1,CKRV,ST}).

The main difficulty that appears when we apply the alteration method in our problem is that  we have to prevent the destruction of all the representations of infinitely many integers $n$. So we need to prove that the number of removed elements  involved in each representation is not too large.

As far as Theorem \ref{Main2} is concerned, the standard application of the probabilistic method also proves that if $\gamma > \frac{g+2}{2g+2}$ then with probability $1$ a sequence in $\mathcal S(\gamma)$ is a $B_2[g]$ sequence (after we remove a finite number of elements). Thus, if $\frac{5}{8}<\gamma<\frac 23$, a random sequence in $\mathcal S(\gamma)$ is, with probability $1$, simultaneously a $B_2[3]$ sequence and  an asymptotic basis of order three.
This result appears in \cite{AS} \S 8.6. To get a $B_2[2]$ basis of order 3 we need to use a more involved argument.

In section \S 3 we explain in more detail the strategy of the proofs and the new ingredients we introduce: the vectorial sunflowers and  the use of a modular Sidon basis.

The Sunflower Lemma was discovered by Erd\H os and Rado \cite{ERa} and has many applications.
In the probabilistic method it has been used to deal with  dependent events when each event can be identified with a set. In our proofs it is more convenient to identify each event with a vector and then we have to use a vectorial version of the Sunflower Lemma. We refer to \cite{AA} for a recent study of other variants of sunflowers.

The modular Sidon basis are used as a trick to simplify  the casuistry in the computations of the random variables appearing in the proofs. See section \S 3 for a more detailed explanation.

The proofs of Theorems \ref{Main2} and \ref{Main1} are quite similar, except that the last one is technically a little more involved.
We prove them in sections \S 4 and \S 5 respectively, however we have left to the last section all the boring calculations of the expected values of the random variables appearing in the proofs.
\subsection{Note added on  April 23, 2013}  Kiss, Rozgonyi and  Sandor \cite{KRS} have proved the existence of a Sidon sequence which is an asymptotic basis of order 4. Their result, which appeared in Arxiv only one day after the first version \cite{C2} of our preprint, has been  obtained independently from our work. They use the method of Kim and Vu \cite{KV} to control the concentration of sums of dependent variables.

\subsection{General Notation} Through the paper we will use the following notation:

 $\bullet$ $f(n)\gg g(n)$ means that there exists $C>0$ such that $f(n)>Cg(n)$ for $n$ large enough. We observe that this includes the possibility that $f(n)=0$ for a finite number of positive integers $n$.

 $\bullet$ $f(n)\ll g(n)$ means that there exists $C>0$ such that $f(n)<Cg(n)$ for all $n$.

 $\bullet$ $f(n)=o(g(n))$ means that $f(n)/g(n)\to 0$ as $n\to \infty$.

 $\bullet$ We write $o_m(1)$ to mean a quantity tending to $0$ as $m\to \infty$.

\section{The modular version of the conjecture.}
Our first theorem about modular Sidon basis is essentially contained in Theorem \ref{modular2}. However it includes the extra condition that $s_1,s_2,s_3$ are pairwise distinct, which is convenient to be used in the proof of Theorem \ref{Main2}. Furthermore the proof is short and has an amusing relation with elliptic curves.
\begin{thm}\label{modular}There exist infinitely many cyclic groups $\Z_N$ containing  a Sidon set $S\subset \Z_N$  and such that any element $x\in \Z_N$ can be written in the form \begin{equation}\label{rep}x=s_1+s_2+s_3,\quad s_i\in S\end{equation} with $s_1,s_2,s_3$ pairwise distinct.
\end{thm}
\begin{proof} Ruzsa \cite{Ru} observed that for all prime $p$ and $g$ a generator of $\mathbb F_p^*$, the set $$S=\{(x,g^x):\ x=0,\dots,p-2\}$$ is a Sidon set in $\Z_{p-1}\times\Z_p.$ Since $\Z_{p-1}\times \Z_p$ is isomorphic to $\Z_{(p-1)p}$ the set $S$ provides an easy construction of a dense Sidon set in a cyclic group. We will prove that $S$ is also a basis of order $3$. In other words, that any element $(a,b)\in \Z_{p-1}\times\Z_p$ can be written as \begin{equation}\label{abr}(a,b)=(x_1,g^{x_1})+(x_2,g^{x_2})+(x_3,g^{x_3}).\end{equation} Indeed we will prove that the number of solutions of \eqref{abr} is exactly the number of points $(U,V),\ V\ne 0$ of the elliptic curve $U^2=4V^3+ (bV+g^a)^2$ in $\mathbb F_p$.

We count, for any $(a,b)\in\Z_{p-1}\times\Z_p$, the number of solutions $(x_1,x_2,x_3)$ of the system
\begin{eqnarray}\label{xxx}
x_1+x_2+x_3&\equiv& a\pmod{p-1}\\
g^{x_1}+g^{x_2}+g^{x_3}&\equiv& b\pmod p.
\end{eqnarray}
This can be written as $$g^{x_1}+g^{x_2}+g^{a-x_1-x_2}\equiv b\pmod p,$$ which is also equivalent to
\begin{equation}\label{XY}X+Y+\frac{\lambda}{XY}\equiv b \pmod p \end{equation} after the change $g^{x_1}=X,\ g^{x_2}=Y,\ g^a=\lambda$.
Now we make another change of variables:
\begin{eqnarray*}
X=\frac{2V^2}{U-bV-\lambda},\qquad
Y=-\frac{\lambda}{V}.
\end{eqnarray*}
Since $XY\ne 0$ we have to add the condition $V\ne 0,\ U\ne bV+\lambda$. With these restrictions the change of variables is bijective.
Applying this lucky change of variables in \eqref{XY} we have
\begin{eqnarray*}
\frac{2V^2}{U-bV-\lambda}-\frac{\lambda}V-\frac{U-bV-\lambda}{2V}&\equiv& b\pmod p\\
\iff \qquad \qquad \qquad\frac{2V^2}{U-bV-\lambda}&\equiv &\frac{U+bV+\lambda}{2V}\pmod p\\
\iff \qquad \qquad 4V^3+ (bV+\lambda)^2&\equiv& U^2\pmod p.
\end{eqnarray*}
Each point of this elliptic curve (except $(U,V)=(\pm \lambda,0)$) corresponds to a solution $(X,Y)$ of \eqref{XY}. By Hasse's Theorem \cite{H}  we know that the elliptic curve has $p+O(\sqrt p)$ points $(U,V)$.

 We  have to remove the solutions $(x_1,x_2,x_3)$ of \eqref{xxx} such that $x_i=x_j$ for some $i\ne j$. Suppose that $x_1=x_2$. In that case the equation \eqref{XY} is $2X+\frac{\lambda}{X^2}\equiv b\pmod p$, which is a cubic equation having at most three solutions. Thus, the number of solutions $(x_1,x_2,x_3)$ of \eqref{xxx} with some repeated coordinates is at most $9$ and the number of representations of $(a,b)$ as a sum of three pairwise distinct elements of $S$ is $p+O(\sqrt p)$.
 \end{proof}
Corollary \ref{modular3} is a byproduct of the proof above and  will be used  in the proof of Theorem \ref{Main1}.
\begin{cor}\label{modular3}There exist infinitely many cyclic groups $\Z_N$ containing  a Sidon set $S\subset \Z_N$  and such that any element $x\in \Z_N$ can be written in the form $$x=s_1+s_2+s_3+s_4,\quad s_i\in S$$ with $s_1,s_2,s_3,s_4$ pairwise distinct.
\end{cor}
\begin{proof}
We will see that the set $S\subset \Z_{p-1}\times \Z_{p}\simeq \Z_{(p-1)p}$ described in Theorem \ref{modular} satisfies the conditions of Corollary \ref{modular3}. From the proof of Theorem \ref{modular} we know that   the number of representations of $(a,b)$ as
\begin{equation}
(a,b)=(x_1,g^{x_1})+(x_2,g^{x_2})+(x_3,g^{x_3})+(0,1),\quad x_i\ne x_j,\ 1\le i<j\le 3
\end{equation}
is $p+O(\sqrt p)$. We observe that all these representations of $(a,b)$ satisfy the conditions of Corollary \ref{modular3} except those  with $x_i=0$ for some $i=1,2,3$.  In theses cases the equation \eqref{XY} is a quadratic equation and the number of these special representations is at most $6$ for each $(a,b)$.
\end{proof}
\subsection{Proof of Theorem \ref{modular2}}
We need the following weaker version  of a result  of Granville, Shparlinski and Zaharesku \cite{G}:
\begin{thm}\label{GSZ}Let $\mathcal C$ be a curve of degree $d$ in $\mathbb F_p^r$ which is absolutely irreducible  in $\mathbb A^r(\overline{\mathbb F}_p)$. Let $h:\mathcal C \rightarrow \mathbb A^s(\overline{\mathbb F}_p)$ be a function $h(X)=(h_1(X),\dots, h_s(X))$ where $h_i(X),\ i=1,\dots,r$ are polynomial functions.

Assume also that there exists $L=L(p)\to \infty$ such that $c_1=\cdots=c_s=0$ whenever $|c_i|\le L,\ i=1,\dots,s$ and $c_1h_1(X)+\cdots +c_sh_s(X)$ is constant along $\mathcal C$.

Under these conditions, the set
$$S=\left \{\left (\frac{h_1(X)}p,\dots, \frac{h_s(X)}p    \right ):\ X\in \mathcal C\right \}  $$
is well distributed in $\mathbb T^s$ when $p\to \infty$.
\end{thm}
Proposition \ref{g} is a consequence of Theorem \ref{GSZ}.
\begin{prop}\label{g}Let $p\equiv 1\pmod 3$ be a prime. For any integers $r_1,r_2$, let $\mathcal C_{r_1,r_2}$ be the curve in $\mathbb F_p^2$ defined by \begin{eqnarray}\label{cu}
x_1^2+x_2^2+(x_1+x_2-r_1)^2&\equiv &r_2\pmod p.
\end{eqnarray}
The set \begin{equation}\label{Set}S_{r_1,r_2}=\left \{\left (\frac{(x_1)_p}p,\frac{(x_2)_p}p,\frac{(x_1^2)_p}{p},\frac{(x_2^2)_p}p\right ):\ (x_1,x_2)\in \mathcal C_{r_1,r_2}\right \}\end{equation} is well distributed in  $[0,1]^4$ when $p\to \infty$. In particular, given $c>0$, any box $B\subset [0,1]^4$ of size $|B|>c$ contains an element of $S_{r_1,r_2}$ if $p$ is large enough.
\end{prop}
\begin{proof}
The equation \eqref{cu} can be written as
\begin{equation}\label{curve} 3(2x_1+x_2-r_1)^2+(3x_2-r_1)^2\equiv 6r_2-2r_1^2\pmod p.\end{equation}
Thus, the curve is absolutely irreducible if $6r_2-2r_1^2\not\equiv 0 \pmod p$.

In the first place we consider the case $6r_2-2r_1^2\not \equiv 0\pmod p$.

We will prove that $L=\sqrt p/3$ works for the condition of Theorem \ref{GSZ}. Suppose that \begin{equation}\label{linear}c_1x_1+c_2x_2+c_3x_1^2+c_4x_2^2=c_0\end{equation} for all $(x_1,x_2)\in \mathcal C_{r_1,r_2}$.
From \eqref{cu} we have that $$x_1^2=x_1(r_1-x_2)-x_2^2+r_1x_2+\frac{r_2-r_1^2}2.$$
Substituting $x_1^2$ in \eqref{linear} by this expression we have that
\begin{eqnarray*}c_1x_1+c_2x_2+c_3\left (x_1(r_1-x_2)-x_2^2+r_1x_2+\frac{r_2-r_1^2}2  \right )+c_4x_2^2=c_0. \end{eqnarray*}
which is equivalent to
$$(c_1+c_3(r_1-x_2))  x_1=(c_3-c_4)x_2^2-(c_3r_1+c_2)x_2+c_0+\frac{c_3(r_1^2-r_2)}2.$$
We write this in a short way as
$P(x_2)x_1=Q(x_2)$ with
\begin{eqnarray*}
P(x_2)&=&-c_3x_2+  c_1+c_3r_1.\\
Q(x_2)&=&(c_3-c_4)x_2^2-(c_3r_1+c_2)x_2+c_0+\frac{c_3(r_1^2-r_2)}2.
\end{eqnarray*}
Multiplying \eqref{linear} by $4c_3$ and completing squares we have
\begin{eqnarray*}(2c_3x_1+c_1)^2+4c_2c_3x_2+4c_3c_4x_2^2=4c_3c_0+c_1^2.\end{eqnarray*}
Multiplying by $P(x_2)^2$ and using that $P(x_2)x_1=Q(x_2)$ we have
\begin{eqnarray*}(2c_3Q(x_2)+c_1P(x_2))^2+P(x_2)^2\left (4c_2c_3x_2+4c_3c_4x_2^2-4c_3c_0-c_1^2\right )=0.\end{eqnarray*}

 This equality must be satisfied for all $x_2$ corresponding to a point $(x_1,x_2)\in \mathcal C_{r_1,r_2}$.
 Since the left hand of the equality above is a polynomial in $x_2$ of degree less than or equal to $4$, this can only be possible if  it is the zero polynomial.
 It is easy to check that the coefficient of $x_2^4$ in the polynomial is $$4c_3^2(c_3-c_4)^2 +4c_3^3c_4=4c_3^2(c_3^2+c_4^2-c_3c_4).$$
If $c_3\ne 0$ we have that $c_3^2+c_4^2-c_3c_4\equiv 0\pmod p$. The inequality
$$\left |c_3^2+c_4^2-c_3c_4\right |\le 3L^2<p$$
 implies that $c_3^2+c_4^2-c_3c_4=0\implies c_3=c_4=0$. Thus $c_3=0$ in any case.

Since $x_1$ and $x_2$ play the same role, we can proceed in the same way to deduce that $c_4=0$. Now we have to consider the possibility that
$c_1x_1+c_2x_2=c_0$ for any $(x_1,x_2)\in \mathcal C_{r_1,r_2}$. But this means that all the solutions of the curve $\mathcal C_{r_1,r_2}$ lie on that line, which is impossible unless $c_0=c_1=c_2=0$.

We have proved that the conditions of Theorem \ref{GSZ} are satisfied when $6r_2- 2r_1^2\not \equiv 0\pmod p$ and then the sets $S_{r_1,r_2}$ are well distributed in this case.

\

Assume now that $6r_2- 2r_1^2\equiv 0\pmod p$.

We observe that in this case the curve \eqref{curve} is not absolutely irreducible.
Let $\omega$ be a solution of $\omega^2+\omega+1\equiv 0 \pmod p$, which exists because $p\equiv 1 \pmod 3$.
It is easy to check that the points $(x_1,x_2)$ of \eqref{curve} are those satisfying
either $x_1=\omega x_2+(1-\omega)r_1$ or $x_1=-\omega x_2+(1+\omega)r_1$ and the curve \eqref{curve} is the union of the two lines:
\begin{eqnarray*}\mathcal C^+_{r_1}&=&\{(x_1,x_2):\ x_1=\omega x_2+(1-\omega)r_1:\ x_1,x_2\in \mathbb F_p\}\\
\mathcal C^-_{r_1}&=&\{(x_1,x_2):\ x_1=-\omega x_2+(1+\omega)r_1:\ x_1,x_2\in \mathbb F_p\}.
\end{eqnarray*} We use again Theorem \ref{GSZ} to prove that the set \eqref{Set} is well distributed when $(x_1,x_2)$ belongs to both curves.
We will do the work for the first one (for the second one the task is similar).

It is clear that $\mathcal C^+_{r_1}$ is absolutely irreducible because it has degree 1. We will prove that $L=\sqrt p/3$ satisfies the condition of Theorem \ref{GSZ}.

Suppose that there exist constants $c_0,c_1,c_2,c_3,c_4$ such that
$$c_1x_1+c_2x_2+c_3x_1^2+c_4x_2^2=c_0$$ for all $(x_1,x_2)\in \mathcal C_{r_1}^+$. In this case we would have
$$c_1(\omega x_2+(1-\omega)r_1)+c_2x_2+c_3(\omega x_2+(1-\omega)r_1)^2+c_4x_2^2=c_0$$ for all $x_2\in \mathbb F_p$ which is not possible if
the coefficient of $x_2^2$ is not $0$. So  $c_4=-c_3\omega^2$ and  using that $\omega^2+1=-\omega$ we have that
$c_4-c_3=-c_3\omega^2-c_3=c_3\omega$. Thus
$$(c_4-c_3)^2+c_3(c_4-c_3)+c_3^2=(c_3\omega)^2+c_3^2\omega+c_3^2=c_3^2(\omega^2+\omega+1)\equiv 0\pmod p.$$
On the one hand we know that $$|(c_4-c_3)^2+c_3(c_4-c_3)+c_3^2|\le (2L)^2+L(2L)+L^2\le 7L^2<p.$$
This implies that $(c_4-c_3)^2+c_3(c_4-c_3)+c_3^2=0$ which implies that $c_3=c_4-c_3=0$ and then $c_3=c_4=0$.

Then the coefficient of $x_2$ must be also $0$, so $c_2=-c_1\omega$ and we have that
$$c_2^2-c_1c_2+c_1^2=c_1^2\omega^2+c_1^2\omega+c_1^2=c_1^2(\omega^2+\omega+1)\equiv 0\pmod p.$$
On the other hand $$\left |c_2^2-c_1c_2+c_1^2\right |\le 3L^2<p.$$
This implies that $c_2^2-c_1c_2+c_1^2=0$ which implies that $c_1=c_2=0$. So $c_0=\cdots =c_4$.

We have proved that the conditions of Theorem \ref{GSZ} are also satisfied when $6r_2- 2r_1^2\equiv 0\pmod p$ for $\mathcal C_{r_1}^+$ (and similarly for $\mathcal C_{r_1}^-$) and then the set $S_{r_1,r_2}$ is well distributed in all the cases.
\end{proof}

\subsection{End of the proof of Theorem \ref{modular2}}Erd\H os and Turan \cite{ETu} showed that the set $$A=\{x+(x^2)_p(2p):\ x=0,\dots, p-1\}$$
is a Sidon set of integers for any odd prime $p$.

For given $N$ let $p$ be a prime such that $p\equiv 1\pmod 3$ and $4p^2<N<5p^2$. This prime exists if $N$ is large enough.
Since $A\subset [0,2p^2)\subset [0,N/2)$, the set $A$  is  a Sidon set in $\Z_N$. We will prove that $A$ is also a basis of order 3 in $\mathbb Z_N$.

We observe that for any integer $K$, the set of integers of the form
\begin{equation}\label{r12}r_1+r_2(2p),\quad K\le r_1,r_2\le  \frac{5p-1}2+K\end{equation}
covers an interval of length $5p^2$. This is clear for $K=0$ and, by translation, for all $K$.
Since $5p^2>N$, in order to prove that $A$ is a basis of order 3 in $\Z_N$ it is enough to prove that any element of the form \eqref{r12} can be written as a sum of $3$ elements of $A$.
We will take $K=\lceil p/4\rceil$ through the proof.

For each $(r_1,r_2)$ we consider the box $B_{r_1,r_2}\subset [0,1]^4$ of all points $(y_1,y_2,y_3,y_4)$ satisfying  the restrictions
\begin{eqnarray*}
\left |y_{1}-\frac{r_1}{3p}\right |,\
\left |y_{2}-\frac{r_1}{3p}\right |,\
\left |y_{3}-\frac{r_2}{3p}\right |,\
\left |y_{4}-\frac{r_2}{3p}\right |\le \frac{K}{12p}.
\end{eqnarray*}
We have to check that $0<y_i<1,\ i=1,\dots,4$ and then that $B_{r_1,r_2}\subset [0,1]^4$. Indeed, since $p\ge 7$, we have
$$  y \le \frac{r_i}{3p}+\frac{K}{12p}\le \frac{\frac{5p-1}2+K}{3p}+\frac{K}{12p}<\frac{5p-1}{6p}+\frac{5K}{12p}\le \frac{5p-1}{6p}+\frac{5(p+3)}{48p}\le \frac{45p+7}{48p}<1$$
and
$y\ge   \frac{r_1}{3p}-\frac{K}{12p}\ge \frac{K}{3p}-\frac{K}{12p}>0$.

The size of this box is $|B_{r_1,r_2}|\ge \left (\frac{K}{12p}\right )^{4}>48^{-4}$ and then Proposition \ref{g} implies that for $p$ large enough  there exists an element, say $\left (\frac{x_1}p,\frac{x_2}p,\frac{(x_1^2)_p}{p},\frac{(x_2^2)_p}p\right )$, with $0\le x_1,x_2\le p-1,\ (x_1,x_2)\in \mathcal C_{r_1,r_2}$ satisfying
\begin{eqnarray*}
\left |\frac{x_{1}}p-\frac{r_1}{3p}\right |,\
\left |\frac{x_{2}}p-\frac{r_1}{3p}\right |,\
\left |\frac{(x_{1}^2)_p}p-\frac{r_2}{3p}\right |,\
\left |\frac{(x_{2}^2)_p}p-\frac{r_2}{3p}\right |\le \frac{K}{12p}.
\end{eqnarray*}
Since $(x_1,x_2)\in \mathcal C_{r_1,r_2}$ there exists an integer $x_3,\ 0\le x_{3}\le p-1$ satisfying
 \begin{eqnarray*}
 x_1+x_2+x_{3}&\equiv &r_1\pmod p\\
  x_1^2+x_2^2+x_{3}^2&\equiv &r_2\pmod p.
 \end{eqnarray*}
Let $m$ be such that $x_1+x_2+x_3=r_1+mp$. We have
\begin{eqnarray*}|m|&\le &\left |\frac{x_1}p-\frac{r_1}{3p}\right |+\left |\frac{x_2}p-\frac{r_1}{3p}\right |+\left |\frac{x_3}p-\frac{r_1}{3p}\right |\\ &\le& \frac{K}{12p}+\frac{K}{12p}+\max\left (\frac{r_1}{3p},1-\frac{r_1}{3p}   \right )\\ &\le &\frac{K}{6p}+\max\left (\frac{5p/2+K}{3p},1-\frac{K}{3p}   \right )\\
&\le &\max \left (\frac{5p+3K}{6p},1-\frac{K}{6p}\right )<1,
\end{eqnarray*}
since $K=\lceil \tfrac p4\rceil$ and $p\ge 7$. This proves that indeed $x_1+x_2+x_3=r_1$. The same argument proves that $(x_1^2)_p+(x_2^2)_p+(x_{3}^2)_p=r_2$.
Thus we have
$$r_1+r_2(2p)=x_1+(x_1^2)_p(2p)+x_2+(x_2^2)_p(2p)+x_3+(x_3^2)_p(2p),$$
that is what we wanted to prove.
\section{The probabilistic method with some new tools} The proofs of Theorems \ref{Main2} and \ref{Main1} are based on the probabilistic method introduced by Erd\H os and Renyi \cite{ER} to study sequences satisfying certain arithmetic properties. The book \cite{AS} is the most complete  reference on the probabilistic method and \cite{HR} is a classic reference for the probabilistic method applied to sequences of integers.

For a given $\gamma$, with $\ 0<\gamma<1$, Erd\H os and Renyi introduced the probabilistic space $\mathcal S(\gamma)$ of all sequences of positive integers $A$ such that all the events $x\in A$ are independent and $\Po(x\in A)=x^{-\gamma}$.

Generally speaking the goal is to prove that a sequence $A$ in $\mathcal S(\gamma)$ satisfies certain arithmetic property (or properties) with high probability. To be more precise, we consider  certain families  $\Omega_n$ of sets of positive integers and the random families $$\Omega_n(A)=\{\omega\in \Omega_n:\ \omega\subset A\}$$ generated by a random sequence $A$ in $S(\gamma)$. Typically we are interested in the random variable $$X_n(A)=|\Omega_n(A)|=\sum_{\omega\in \Omega_n}I(\omega\subset A).$$
For example if \begin{equation}\label{ex3}\Omega_n=\{\omega=\{x_1,x_2,x_3\}:\ x_1+x_2+x_3=n\},\end{equation} the random variable $X_n(A)$ counts the number of representation of $n$ as a sum of three elements of a random sequence $A$ in $\mathcal S(\gamma)$. In general we are interested in proving that $X_n(A)$  satisfies a certain property $P_n$.
The standard strategy is to first prove that \begin{equation}\label{Borel}\sum_n\Po(X_n(A)\text{ does not satisfies }P_n)<\infty\end{equation} and then apply Borel-Cantelli lemma to deduce that with probability $1$, the random variable $X_n(A)$ satisfies property $P_n$ for all $n$ large enough.

We will modify the probabilistic space $\mathcal S(\gamma)$ to force that all the elements of $A$ lie in some residue classes $s\in S\pmod N$ for some $S\subset \Z_N$ satisfying suitable conditions. At the end of this section  we explain the advantage of this modification. We will write $x\equiv S\pmod N$ to mean that $x\equiv s\pmod N$ for some $s\in S$.

Also it is technically more convenient to introduce a parameter $m$ to force that  the elements  of $A$ are greater than a fixed $m$.   This  idea was introduced before in \cite{CKRV} and allows us to bound \eqref{Borel} by a quantity which is $o_m(1)$. At a later step we take $m$ as large as we want.
\begin{defn}Let $S$ be a non empty set of a cyclic group $\Z_N$. For a given $\gamma,\ 0<\gamma<1$ and a given positive integer $m$, let $\mathcal S_m(\gamma;\ {\scriptstyle S\mod N} )$ be the probabilistic space of all sequences of positive integers $A$ such that all the events $x\in A$ are independent and such that $$\Po(x\in A)=\begin{cases}x^{-\gamma}\quad\text{  if  }\quad x\equiv S\pmod N \text{ and } x>m\\ 0 \qquad  \text{ otherwise}.\end{cases}$$
\end{defn}
Since $X_n(A)$ is a sum of boolean variables we expect that $X_n(A)$ is  concentrated around its expected value, $\mu_n=\E(X_n(A))$, with high probability.

When the variables $I(\omega\in A)$ are independent (the sets $\omega \in \Omega_n$ are disjoint), Chernoff's theorem is enough to prove that $X_n(A)$ is strongly concentrated around $\mu_n$. However, when the sets in $\Omega_n$ are not disjoint, as in the example \eqref{ex3}, the study of the concentration is more involved.

It is expected, however, that if the dependent events have small correlation we still have enough concentration.
Janson's inequality \cite{J} serves our purpose for the lower tail:
 \begin{thm}[Janson's inequality]
Let $\Omega$ be a family of sets  and let $A$ be a random subset. Let $X(A)=|\{\omega\in \Omega:\ \omega\subset A\}|$ with finite expected value $\mu=\E(X(A))$. Then
$$\Po(X\le (1-\varepsilon)\mu )\le exp\left (-\varepsilon^2\mu^2/(2\mu+\Delta(\Omega))  \right )$$ where
$$\Delta(\Omega)=\sum_{\substack{\omega,\omega'\in \Omega\\ \omega\sim \omega'}}\Po(\omega,\omega'\subset A)$$ and $\omega\sim \omega'$ means that $\omega\cap \omega'\ne \emptyset$ and $\omega\ne \omega'$.
In particular, if $\Delta(\Omega)<\mu$ we have that
$$\Po (X\le \mu/2)\le \exp \left (-\mu/12  \right ).   $$
 \end{thm}
To deal with the upper tail  Erd\H os and Tetali \cite{ET} introduced the Sunflowers trick.
\subsection{Sunflowers and vectorial Sunflowers}
A  collection of sets $S_1,\dots, S_k$ forms a sunflower if there exists a set $C$ such that $S_i\cap S_j=C$ for any $i\ne j$. The sets $S_i\setminus C$ are the petals and $C$ is the core of the sunflower. Erd\H os and Rao \cite{ERa} proved the following interesting lemma.
\begin{lem}[Sunflower lemma]\label{Sunflower} Let $\Omega$ a family of $h$-sets.  If $\Omega$ does not contain a  sunflower  of $k$ petals then $|\Omega|\le h!(k-1)^h$.
\end{lem}
We will work with a variant of the Sunflower lemma which deals with vectors instead of sets. The reason is that  in our proofs sometimes it will be more convenient to work with families $\Omega$ of vectors (instead of sets).
\begin{defn}For a given vector $\overline x=(x_1,\dots,x_h)$ we define $\text{Set}(\overline x)=\{x_1,\dots,x_h\}$. We say that a collection of $k$  distinct vectors $\overline x_j,\ j=1,\dots ,k$ forms a  disjoint set of $k$ vectors ($k$-d.s.v. for short) if $ \text{Set}(\overline x_j)\cap \text{Set}(\overline x_{j'})=\emptyset $ for any $j\ne j'$.
\end{defn}
\begin{defn}We say that $k$  distinct vectors with $h$ coordinates form a  vectorial sunflower (of $k$ petals) if for some $I\subset [h]$ the following two conditions are satisfied:
\begin{itemize}
  \item For all $i\in I$ all the vectors have the same i-th coordinate.
  \item The set of vectors obtained by removing all the i-th coordinates, $i\in I$, forms a $k$-d.s.v.
\end{itemize}
We say that the vectorial sunflower is of type $I$. \end{defn}
We observe that a vectorial sunflower (of $k$ petals)  of type $I=\emptyset$ is a $k$-d.s.v. The example belove forms a vectorial sunflower (of $4$ petals) of  type $I=\{2,5\}$.
\begin{eqnarray*}
\overline{x}_1&=&(\ 7, \ \ \textbf{7},\ \ 1,\   13,\  \textbf{8})\\
\overline{x}_2&=&(17, \ \ \textbf{7},\ \ 6, \ \ 6,\  \textbf{8})\\
\overline{x}_3&=&(\ 8, \ \ \textbf{7},\  18,\ \  8,\  \textbf{8})\\
\overline{x}_4&=&(11, \ \ \textbf{7},\ \ 4,\ \ 5,\  \textbf{8})
     \end{eqnarray*}
We need a vectorial version of Lemma \ref{Sunflower}.
\begin{lem}[Vectorial sunflower lemma]\label{sunflower} Let $\Omega$ be a family of vectors of $h$ coordinates.  If $\Omega$ does not contain a  vectorial sunflower of $k$ petals  then $|\Omega|\le h!((h^2-h+1)k)^h$.
\end{lem}
\begin{proof}
Suppose that $|\Omega|>h!((h^2-h+1)(k-1))^h$. For any $\overline x=(x_1,\dots,x_h)\in \Omega$ we consider the set $\text{Set}_h(\overline x)=\{hx_1+1,hx_2+2,\dots, hx_h+h\}$ and the family $\hat \Omega=\{\text{Set}_h(\overline x):\ \overline x\in \Omega\}$. The sunflower lemma of Erd\H os-Rao applied to $\hat \Omega$  implies that there exists a classical sunflower with $(h^2-h+1)(k-1)+1$ petals, say $\text{Set}_h(\overline x_1),\dots ,\text{Set}_h(\overline x_{(h^2-h+1)(k-1)+1})$. It is clear that from these sets we can recover the corresponding  vectors  $\overline x_1,\dots ,\overline x_{(h^2-h+1)(k-1)+1}$ which satisfy the following conditions:
\begin{itemize}
  \item There exists $I\subset \{1,\dots, h\}$ such that for each $i\in I$ all the $(h^2-h+1)(k-1)+1$ vectors have the same $i$-th coordinate.
  \item For each $i\not \in I$, the $i$-th coordinates of all these vectors are pairwise distinct.
\end{itemize}
We observe that the conditions above are not enough to make sure that the vectors form a vectorial sunflower.
We will proof, however, that  the set $\{\overline x_1,\dots ,\overline x_{(h^2-h+1)(k-1)+1}\}$ contains a vectorial sunflower of $k$ petals.

Select  one vector, say $\overline x_1$. We know that if $i\not \in I$, the $i$-th coordinate of $\overline x_1$ cannot be equal to the $i$-th coordinate of a distinct vector. However it may be equal to a different $i'$-th coordinate ($i'\not \in I$) of a distinct vector. We observe that for each $i\not \in I$ and for each $i'\not \in I,\ i'\ne i$ there is at most one such vector.

We remove, for each $i\not \in I$ and for each $i'\not \in I ,\ i'\ne i$, that vector (if it exists). Thus removing at most $h(h-1)$ vectors we make sure that for all $i\not \in I$, the $i$-th coordinate of $x_1$ is not equal to any $i'$-th coordinate ($i'\not \in I,\ i'\ne i$) of a distinct vector.

 Now we select a second vector and proceed as above. Since the number of original vectors was $(h(h-1)+1)(k-1)+1$  we can select at least $k$ vectors in this way forming a vectorial sunflower of $k$ petals.
\end{proof}
Typically we will deal with families of vectors $\Omega$ and with the corresponding random families $\Omega(A)=\{\overline x\in \Omega:\ \text{Set}(\overline x)\subset A\}$.
\begin{cor}\label{KV} Let $\Omega_n$ be a sequence of families of vectors of $h$ coordinates. Suppose that with probability $1-o_m(1)$ the random families $\Omega_n(A)$ do not contain vectorial sunflowers of $K$ petals for any $n$. Then, with probability $1-o_m(1)$ it holds that $|\Omega_n(A)|\le h!((h^2-h+1)K)^h$ for all $n$.
\end{cor}
The following proposition will be used several times in the proofs of Theorems \ref{Main2} and \ref{Main1}.
\begin{prop}\label{ind}Let $\{\Omega_n\}$ be a sequence of families of vectors and $\{\Omega_n(A)\}$ the corresponding random family  where $A$ is a random sequence in $\mathcal S_m(\gamma, {\scriptstyle S\pmod N})$.  Suppose that there is $\delta>0$ such that $\E(|\Omega_n(A)|)\ll (n+m)^{-\delta}$. If $K>1/\delta$ then
$$\Po(\Omega_n(A) \text{ contains a K-d.s.v. for some } n)=o_m(1).$$
\end{prop}
\begin{proof}
\begin{eqnarray*}
\Po\left (\Omega_n(A) \text{ contains a K-d.s.v.}\right )&\le &\sum_{\substack{\overline x_1,\dots,\overline x_K\in \Omega_n\\ \text{ form a } K\text{-d.s.v.}}}\Po\big (\text{Set}(\overline x_1),\dots, \text{Set}(\overline x_K)\subset A\big )\\
&=&\sum_{\substack{\overline x_1,\dots,\overline x_K\in \Omega_n\\ \text{ form a } K\text{-d.s.v.}}}\Po\big (\text{Set}(\overline x_1)\subset A\big)\cdots \Po\big( \text{Set}(\overline x_K)\subset A\big)\\
&\le &\frac 1{K!}\left (\sum_{\overline x\in \Omega_n}\Po\big(\text{Set}(\overline x)\subset A\big )   \right )^K\\
& = &\frac{\E(|\Omega_n(A)|)^K}{K!}\ll \frac{(n+m)^{-\delta K}}{K!}.
\end{eqnarray*} Then,
\begin{eqnarray*}\Po(\Omega_n(A) \text{ contains a K-d.s.v. for some } n)&\ll & \sum_{n}\Po\left (\Omega_n(A) \text{ contains a K-d.s.v. }\right )\\ &\ll & \sum_n\frac{(n+m)^{-\delta K}}{K!}=o_m(1).\end{eqnarray*}
\end{proof}

\subsection{The modular trick}
Before we explain  the strategy of the proof of Theorem \ref{Main2} we advance that we will deal with sums of the form
\begin{equation}\label{ss}\sum_{\substack{\overline x=(x_1,\dots,x_8)\\ x_1+x_2+x_3=n\\x_1+x_4=x_5+x_6=x_7+x_8\\ \{x_1,x_4\}\ne \{x_5,x_6\}\ne\{x_7,x_8\}}}\Po(x_1,\dots, x_8\in A)\end{equation} where  $A$ is a random sequence. If the coordinates of a vector $\overline x$ are  pairwise distinct, then $\Po(x_1,\dots, x_8\in A)=\prod_{i=1}^8\Po(x_i\in A)$ and the computation of \eqref{ss} is straightforward. Unfortunately we have also to consider those vectors with repeated coordinates. There are many patterns to consider and the computation of the sum above would be hard in a standard probabilistic space $\mathcal S(\gamma)$.  To reduce this unpleasant task we will restrict the sequences $A$ to be in some residue classes $s\in S,\pmod N$ for some $S\subset \Z_N$ given in Theorem \ref{modular}.  This trick will simplify a lot the casuistry of the possible coincidences between the coordinates in the proofs of Lemmas \ref{Tn} and \ref{B1n}.

\section{$B_2[2]$ sequences which are asymptotic basis of order $3$}
In this section we prove Theorem \ref{Main2}.
 \subsection{Strategy of the proof}We start by fixing a cyclic group $\Z_N$ and a set $S\subset \Z_N$ satisfying the conditions of Theorem \ref{modular}.
Throughout this section  we will consider the probabilistic space $\mathcal S_m(7/11;{\scriptstyle{S\mod N}}).$ Indeed,  it would work  any $\gamma,\quad \frac{5}8<\gamma<\frac 23$.
We consider the sequence of sets $$Q_n=\Big \{\omega=\{x_1,x_2,x_3\}:\ x_1+x_2+x_3=n,\ x_i\not \equiv x_j\pmod N,\ i\ne j\Big \}.$$
Given a sequence of positive integers $A$ we define, for each $n$, the set
$$Q_n(A)=\{\omega\in Q_n:\ \omega\subset A\}.$$
\begin{defn}[Lifting process] The $B_2[2]$-lifting process of a sequence $A$ consists in removing  from $A$ those elements $a_1\in A$ such that there exist $a_2,a_3,a_4,a_5,a_6\in A$ with  $a_1+a_2=a_3+a_4=a_5+a_6$ and $\{a_1,a_2\}\ne \{a_3,a_4\}\ne \{a_5,a_6\}$.

We  denote by $A_{B_2[2]}$ the surviving elements of $A$ after this process. The sequence $A_{B_2[2]}$ clearly  is a $B_2[2]$ sequence.
\end{defn}
We define
\begin{eqnarray*}T_{n}&=&\{\overline x=(x_1,\dots, x_8):\  \overline x\text{ satisfies  } \text{cond}(T_{n})\}\quad \text{ where}
\\ \text{cond}(T_{n})&:=&\begin{cases}\{x_1,x_2,x_3\}\in Q_n \\x_1+x_4=x_5+x_6=x_7+x_8,\qquad \{x_1,x_4\}\ne \{x_5,x_6\}\ne \{x_7,x_8\}
\\ x_1\equiv x_5\equiv x_7\pmod{N},\ x_4\equiv x_6\equiv x_8\pmod{N}.\end{cases}   \end{eqnarray*}
We define also $$T_n(A)=\{\overline x\in T_n:\ \text{Set}(\overline x)\subset A\}.$$
We will see that $|T_{n}(A)|$ is an upper bound for the number of the representations of $n$ counted in $Q_n(A)$ that are destroyed in the $B_2[2]$-lifting process of $A$ defined above.

Suppose that $\omega=\{x_1,x_2,x_3\}\in Q_n(A)$ contains an element, say $x_1$, which is removed in the $B_2[2]$-lifting process. Then there exist $x_4,x_5,x_6,x_7,x_8\in A$ such that $x_1+x_4=x_5+x_6=x_7+x_8$ with $\{x_1,x_4\}\ne \{x_5,x_6\}\ne \{x_7,x_8\}$. On the other hand, since all $x_i\equiv S\pmod N$ and $S$ is a Sidon set in $\Z_N$, interchanging $x_5$ with $x_6$ and $x_7$ with $x_8$ if needed, we have that $x_1\equiv x_5\equiv x_7\pmod N$ and $x_4\equiv x_6\equiv x_8\pmod N$. Thus, any $\omega\in Q_n(A)$ destroyed in the $B_2[2]$-lifting process is counted at least once in $T_n(A)$ and we have
$$|Q_n(A_{B_2[2]})|\ge |Q_n(A)|-|T_{n}(A)|.$$
Since $A_{B_2[2]}$ is a $B_2[2]$ sequence for any sequence $A$, to proof Theorem \ref{Main2} it is enough to prove that there exists a sequence $A$ such that $|Q_n(A)|\gg n^{\delta}$ for some $\delta>0$ and for $n$ large enough and such that  $|T_{n}(A)|\ll 1$. We perform these tasks in Propositions \ref{MP12} and \ref{MP22}.

\begin{prop}\label{MP12} With probability 1 we have $|Q_n(A)|\gg n^{1/11}$ for $n$ large enough.
\end{prop}
\begin{proof}
We apply Janson's inequality to $\Omega=Q_n$ and $X=|Q_n(A)|=\{\omega\in Q_n:\ \omega\subset A\}$ where $A$ is a random sequence in $\mathcal S_m(7/11,{\scriptstyle S\pmod N})$. In Lemma \ref{Qn} we prove that $\mu_n=\E(Q_n(A))\gg n^{1/11}$ and in Proposition \ref{DeltaQn} we prove that $\Delta(Q_n)\ll n^{-2/11}$ for $$\Delta(Q_n)=\sum_{\substack{\omega,\omega'\in Q_n\\ \omega\sim\omega'}}\Po(\omega,\omega'\in A).$$ Thus for $n$ large enough we have that $\Delta_n<\mu_n$ and Janson's inequality implies that $$\Po(|Q_n(A)|\le \mu_n/2)\le \exp\left (-\mu_n/12\right ).$$
Then for some $C>0$ we can write
$$\sum_n\Po(|Q_n(A)|\le \mu_n/2)<\sum_n \exp\left (-Cn^{1/11}\right )<\infty$$ and the Borell-Cantelli lemma implies that with probability $1$ we have $|Q_n(A)|\ge \mu_n/2\gg n^{1/11}$ for all $n$ large enough.
\end{proof}
In the proof of Proposition \ref{MP22} we use several times Lemma \ref{UU1}. We first introduce the following families of vectors, whose expected values are bounded in Lemma \ref{VV1}.
\begin{eqnarray}\label{set1}
\qquad U_{2r}&=&\{\overline x=(x_1, x_2):\ x_1+x_2=r,\ x_1\ne x_2\}\\
\qquad V_{2r}&=&\{\overline x=(x_1,x_2):\ x_1-x_2=r,\ x_1\ne x_2\}\nonumber \\
\qquad W_r&=&\{\overline x=(x_4,x_5,x_6,x_7,x_8):\ x_5+x_6-x_4=x_7+x_8-x_4=r,\ x_i\ne x_j\}.\nonumber
\end{eqnarray}
\begin{lem}\label{UU1}Let $X_r$ be any of the three families in \eqref{set1}. Then
$$\Po(X_r(A) \text{ contains a } 12\text{-d.s.v.} \text{ for some } r)=o_m(1).$$
\end{lem}
\begin{proof}
Lemma \ref{VV1} implies that $\E(|X_r(A)|)\ll (r+m)^{-2/11}$ and then apply Proposition \ref{ind}.
\end{proof}
\begin{prop}\label{MP22} With probability $1-o_m(1)$,  $|T_n(A)|\le 10^{28}$ holds for any $n$.
\end{prop}
\begin{proof} We claim that

\textbf{Claim.}
\emph{With probability $1-o_m(1)$, $T_n(A)$ does not contain vectorial sunflowers of $12$ petals for any $n$.}

Assuming the Claim  we can apply Corollary \ref{KV} to the families $T_n$ to deduce that with probability $1-o_m(1)$, we have that $|T_n(A)|\le 8!((8^2-8+1)12)^8<10^{28}$ for all $n$. Hence the Claim implies Proposition \ref{MP22}.

We prove the Claim for the distinct possible types $I\subset \{1,\dots,8\}$ of the vectorial sunflowers in $T_n(A)$. The types we analyze below will cover all the cases, as we will explain later.

\begin{itemize}
  \item [1.]$I=\emptyset$.   Lemma \ref{Tn}  $\implies \E(|T_n(A)|)\ll (n+m)^{-1/11}$.

 Proposition \ref{ind} $\implies \Po(T_n(A) \text{ has  a $12$-d.s.v.  for some } n)=o_m(1)\implies $

 $\implies$ the Claim holds for vectorial sunflowers of type $I=\emptyset$.

\

  \item [2.]$|I\cap \{1,2,3\}|=1$.   Suppose that $I\cap \{1,2,3\}=\{1\}$. The other two cases are similar.

If  $T_{n}(A)$ contains  a vectorial sunflower (of $12$ petals) of type $I$ for some $n$ (denote by $l_1$ the common first coordinate) then there is a $12$-d.s.v. $\overline x_j=(x_{2j},x_{3j}),\ j=1,\dots ,12$ such that $x_{2j}+x_{3j}=n-l_1$. Thus, for $r=n-l_1$, $U_{2r}(A)$ contains a $12$-d.s.v. and Lemma \ref{UU1} implies the Claim for vectorial sunflowers of this type.

\

  \item [3.]$|I\cap \{1,4,5,6,7,8\}|=1$.
  Suppose that $I\cap \{1,4,5,6,7,8\}=\{1\}$. The other  cases are similar.

If  $T_{n}(A)$ contains  a vectorial sunflower (of $12$ petals) of type $I$ for some $n$ (denote by $l_1$  the common first  coordinate) then there is a $12$-d.s.v. $\overline x_j=(x_{4j},x_{5j},x_{6j},x_{7j},x_{8j}),\ j=1,\dots 12$ such that $x_{5j}+x_{6j}=x_{7j}+x_{8j}=l_1+x_{4j}$. Thus, for  $r=l_1$, $W_{r}(A)$ contains a $12$-d.s.v. and Lemma \ref{UU1} implies  the Claim for vectorial sunflowers of this type.

\

  \item [4.]$|I\cap \{1,4,5,6\}|=2$ or $|I\cap \{1,4,7,8\}|=2$ or $|I\cap \{5,6,7,8\}|=2$.
  Suppose that $|I\cap \{1,4,5,6\}|=2$. The other  cases are similar. We need to distinguish between two essentially distinct cases:
  \begin{itemize}
    \item [i)] $I\cap \{1,4,5,6\}=\{1,4\}$.
    If $T_n(A)$ contains a vectorial sunflower (of $12$ petals) of  type $I$ (denote by $l_1,l_4$  the value of the common coordinates) then there is a $12$-d.s.v. $\overline x_j=(x_{5j},x_{6j}),\ j=1,\dots 12$ such that $x_{5j}+x_{6j}=l_1+l_4$. Thus, for  $r=l_1+l_4$, $U_{2r}(A)$ has a $12$-d.s.v. Lemma \ref{UU1} implies the Claim for vectorial sunflowers of this type.

\

     \item [ii)] $I\cap \{1,4,5,6\}=\{1,5\}$. If $T_n(A)$ contains a vectorial sunflower (of $12$ petals) of type $I$ (denote by $l_1,l_5$  the value of the common coordinates and assume that $l_1>l_5$) we have that there is an $12$-d.s.v. $\overline x_j=(x_{4j},x_{6j}),\ j=1,\dots 12$ such that $x_{6j}-x_{4j}=l_1-l_5$. Thus, for $r=l_1-l_5$, $V_{2r}(A)$ contains a $12$-disjoint set and Lemma \ref{UU1} implies the Claim for vectorial sunflowers of this type.
              \end{itemize}
 \end{itemize}
The sets $I$ considered in the previous analysis cover all the possible cases. The point is that if  the indexes of one of the equations $x_1+x_2+x_3=n$, $\ x_1+x_4=x_5+x_6$, $\ x_1+x_4=x_7+x_8$, $\ x_5+x_6=x_7+x_8$ are all in $I$ except one of them, then a vectorial sunflower of that type $I$ cannot exist. For example the cases such that $|I\cap\{1,2,3\}|=2$ are not possible because if two vectors have the same coordinates $x_1,x_2$, also $x_3$ must be the same in both vectors. For example, if $x_4,x_5,x_6\in I$ then these coordinates must be the same in all the vectors of the sunflower, but also $x_1$ should be the same in all of them. The reader can check that the types not studied above are of this kind. Thus we have proved the Claim.
\end{proof}
\section{Sidon basis of order $3+\varepsilon$} In this section we will prove Theorem \ref{Main1}. The proof follows the same steps than the proof of Theorem \ref{Main2} but is a little more involved because we have to distinguish an element $x_i\le n^ {\varepsilon}$.
\subsection{Strategy of the proof of Theorem \ref{Main1}}We start by fixing a cyclic group $\Z_N$ and a set $S\subset \Z_N$ satisfying the conditions of Theorem \ref{modular3}. Throughout this section  we will consider the probabilistic space $\mathcal S_m(\gamma,{\scriptstyle S\pmod{N}})$ with $$\gamma=\frac 23+\frac{\varepsilon}{9+9\varepsilon}.$$ Indeed we could take any $\gamma$ with $\frac{2+3\varepsilon}{3+4\varepsilon}<\gamma<\frac{2+\varepsilon}{3+\varepsilon}$.
We consider the families of sets
\begin{eqnarray*}
R_n&=&\Big \{ \omega=\{x_1,x_2,x_3,x_4\}:\ \text{ satisfying the conditions cond}(R_n)\Big \}\quad  \text{ where}\\
\text{cond}(R_n)&=&\begin{cases} x_1+x_2+x_3+x_4=n,\\ \min(x_1,x_2,x_3,x_4)\le n^{\varepsilon}
\\ x_i\not \equiv x_j\pmod{N},\ 1\le i<j\le 4.\end{cases}  \end{eqnarray*}

Given a sequence of positive integers $A$ we define the families:
 \begin{eqnarray*}R_n(A)&=&\Big \{\omega\in R_n:\ \omega\subset A\Big \}. \end{eqnarray*}
\begin{defn}[Sidon lifting process] The Sidon lifting process of a sequence $A$ consists in removing  from $A$ those elements $a\in A$ such that there exist $a',a'',a'''\in A$ with $a+a'=a''+a''',\ \{a+a'\}\ne \{a'',a'''\}$.

We  denote by $A_{\text{Sidon}}$ the surviving elements of $A$ after this process.
\end{defn}

We define
\begin{eqnarray*}B_{n}(A)&=&\{\overline x=(x_1,\dots, x_7):\ x_i\in A,\ \overline x\text{ satisfies  } \text{cond}(B_{n})\}\quad \text{ where}
\\ \text{cond}(B_{n})&:=&\begin{cases}\{x_1,x_2,x_3,x_4\}\in R_n \\x_1+x_5=x_6+x_7,\qquad \{x_1,x_5\}\ne \{x_6,x_7\}\\ x_1\equiv x_6\pmod{N},\ x_5\equiv x_7\pmod{N}.\end{cases}   \end{eqnarray*}

By a similar argument as the one used in the proof of Theorem \ref{Main2} we can see that $|B_{n}(A)|$ is an upper bound for the number of the representations of $n$ counted in $R_n(A)$ but destroyed in the Sidon lifting process of $A$. Thus,
$$|R_n(A_{\text{Sidon}})|\ge |R_n(A)|-|B_{n}(A)|.$$

Since $A_{\text{Sidon}}$ is a Sidon sequence, to prove Theorem \ref{Main1} it is enough to prove that there exists a sequence $A$ such  that $|R_n(A)|\gg n^{\delta}$ for some $\delta>0$ and for $n$ large enough, and such that  $|B_{n}(A)|\ll 1$.
\begin{prop}\label{MP1} With probability 1 we have that $|R_n(A)|\gg n^{\frac{2\varepsilon^2}{9+9\varepsilon}}$ for $n$ large enough.
\end{prop}
\begin{proof}
We apply Janson's inequality to $\Omega=R_n$ and $X=|R_n(A)|=\{\omega\in R_n:\ \omega\subset A\}$ where $A$ is a random sequence in $\mathcal S_m(\gamma,{\scriptstyle S\mod N})$. In Proposition \ref{Rn} we prove that $\mu_n=\E(R_n(A))\gg n^{\frac{2\varepsilon^2}{9+9\varepsilon}}$ and in Proposition \ref{Deltan} we prove that $\Delta(R_n)\ll  n^{\frac{-3\varepsilon+2\varepsilon^2}{9+9\varepsilon}}$ for $$\Delta(R_n)=\sum_{\substack{\omega,\omega'\in R_n\\ \omega\sim\omega'}}\Po(\omega,\omega'\in A).$$ Thus for $n$ large enough we have $\Delta(R_n)<\mu_n$ and Janson's inequality implies that $$\Po(|R_n(A)|\le \mu_n/2)\le \exp\left (-\mu_n/12\right ).$$
Then for some $C>0$ we have
$$\sum_n\Po(|R_n(A)|\le \mu_n/2)<\sum_n \exp\left (-Cn^{\frac{2\varepsilon^2}{9+9\varepsilon}}\right )<\infty$$ and the Borell-Cantelli lemma implies that with probability $1$ we have $|R_n(A)|\ge \mu_n/2\gg n^{\frac{2\varepsilon^2}{9+9\varepsilon}}$ for all $n$. This proves Proposition \ref{MP1}.
\end{proof}

In the proof of Proposition \ref{MP2} we use several times Lemma \ref{UUU}. We first introduce the following families of vectors, whose expected values are bounded in Lemma \ref{VV2}.
\begin{eqnarray}\label{set2}
\qquad U_{2r}&=&\{\overline x=(x_1, x_2):\ x_1+x_2=r,\ x_1\ne x_2\}\\
\qquad U_{3r}&=&\{\overline x=(x_1, x_2,x_3):\ x_1+x_2+x_3=r,\ x_i\ne x_j\}\nonumber \\
\qquad V_{2r}&=&\{\overline x=(x_1,x_2):\ x_1-x_2=r,\ x_1\ne x_2\}\nonumber \\
\qquad V_{3r}&=&\{\overline x=(x_1,x_2,x_3):\ x_1+x_2-x_3=r,\ x_i\ne x_j\}.\nonumber
\end{eqnarray}
\begin{lem}\label{UUU}Let $K$ a positive integer such that $K>18/\varepsilon^2$. Then for any of the four families $X_r$ in \eqref{set2}
$$\Po(X_r(A) \text{ contains a } K\text{-d.s.v.} \text{ for some } r)=o_m(1).$$
\end{lem}
\begin{proof}
Lemma \ref{VV2} implies  $\E(|X_r(A)|)\ll (r+m)^{\varepsilon/6}\ll (r+m)^{-\varepsilon^2/18}$ and then apply Proposition \ref{ind}.
\end{proof}

\begin{prop}\label{MP2} With probability $1-o_m(1)$ we have  $ |B_n(A)|\ll 1.$
\end{prop}
\begin{proof}
\textbf{Claim:} \emph{Let $K$ a positive integer such that $K>18/\varepsilon^2$. Then  $B_n(A)$ does not contain vectorial sunflowers of $K$ petals for all $n$, with probability $1-o_m(1).$}

 Assuming the Claim  we can apply Corollary \ref{KV} to the families $B_n$ to deduce that $|B_n(A)|\le 7!((7^2-7+1)K)^7$ for all $n$,
with probability $1-o_m(1)$.

We prove the Claim for the distinct possible type $I\subset \{1,\dots,7\}$ of the vectorial sunflowers in $B_n(A)$. The types we analyze below will cover all the cases. It is clear that if $I\cap \{1,2,3,4\}=\{1,2,3\}$, then vectorial sunflowers of type $I$ cannot exists because the conditions on $B_{n}$ implies that also the 4th-coordinate is common for all vectors. The same argument works for any $I$ such that $|I\cap \{1,2,3,4\}|=3$ or $|I\cap \{1,5,6,7\}|=3$. Also it is clear that there do not exist vectorial sunflowers of type $I=[7]$. Thus we have to consider  the  types: $I=\emptyset$, $|I\cap \{1,2,3,4\}|=1$, $|I\cap \{1,2,3,4\}|=2$, $|I\cap \{1,5,6,7\}|=1$, $|I\cap \{1,5,6,7\}|=2$.

\begin{itemize}
  \item [1.]$I=\emptyset$. Lemma \ref{B1n} $\implies \E(|B_n(A)|)\ll (n+m)^{-\frac{\varepsilon^2}{18}}$.

  Proposition \ref{ind} $\implies \Po(B_n(A) \text{ has  a $K$-d.s.v.  for some } n)=o_m(1)\implies $

  $\implies$ the Claim holds for vectorial sunflowers of type $I=\emptyset$.

  \

  \item [2.]$|I\cap \{1,2,3,4\}|=1$.  Suppose that $I\cap \{1,2,3,4\}=\{1\}$. The other three cases are similar.

If  $B_{n}(A)$ contains  a vectorial sunflower of $K$ petals of this type for some $n$ (denote by $l_1$ to the common first coordinate) we have that there exists an $K$-d.s.v. $\overline x_j=(x_{2j},x_{3j},x_{4j}),\ \text{Set}(\overline x_j)\subset A,\ j=1,\dots, K$ such that $x_{2j}+x_{3j}+x_{4j}=n-l_1$. Thus, for $r=n-l_1$,  $U_{3r}(A)$ contains a $K$-d.s.v. and Lemma \ref{UUU} implies the Claim for vectorial sunflowers of this type.

\
  \item [3.]$|I\cap \{1,2,3,4\}|=2$.  Suppose that $I\cap \{1,2,3,4\}=\{1,2\}$. The other six cases are similar.

If  some $B_{n}(A)$ contains  a vectorial sunflower of $K$ petals of this type  for some $n$ (denote by $l_1,l_2$  the common first and second coordinate) we have that there exists an $K$-d.s.v. $\overline x_j=(x_{3j},x_{4j}),\ \text{Set}(\overline x_j)\subset A,\ j=1,\dots K$ such that $x_{3j}+x_{4j}=n-l_1-l_2$. Thus, for $r=n-l_1-l_2$ we have that $U_{2r}(A)$ contains a $K$-d.s.v. and Lemma \ref{UUU} implies the Claim for vectorial sunflowers of this type.

\

  \item [4.]$|I\cap \{1,5,6,7\}|=1$.  Suppose that $I\cap \{1,5,6,7\}=\{1\}$. The other four cases are similar.

  If some $B_{n}(A)$ contains  a vectorial sunflower of $K$ petals of this type   (denote by $l_1$ the common first  coordinate) we have that there exists an $K$-d.s.v. $\overline x_j=(x_{5j},x_{6j},x_{7j}),\ \text{Set}(\overline x_j)\subset A,\ j=1,\dots K$ such that $x_{6j}+x_{7j}-x_{5j}=l_1$. Thus, for $r=l_1$ we have that $V_{3r}(A)$ contains a $K$-d.s.v. and Lemma \ref{UUU} implies the Claim for vectorial sunflowers of this type.

\

  \item [5.]$|I\cap \{1,5,6,7\}|=2$.  We  distinguish
 two essentially distinct cases:

\begin{itemize}
  \item [i)] $I\cap \{1,5,6,7\}=\{1,5\}.$ The case $I\cap \{1,5,6,7\}=\{6,7\}$ is similar.

  If  some $B_{n}(A)$ contains  a vectorial sunflower of $K$ petals of this type  (denote by  $l_1,l_5$  the common first and 5 coordinate)  we have that there exists an $K$-d.s.v. $\overline x_j=(x_{6j},x_{7j}),\ \text{Set}(\overline x_j) \subset A,\ j=1,\dots K$ such that $x_{6j}+x_{7j}=l_1+l_5$. Thus, for $r=l_1+l_5$ we have that $U_{2r}(A)$ contains a $K$-d.s.v. and Lemma \ref{UUU} implies the Claim for vectorial sunflowers of this type.

  \

  \item [ii)]$I\cap \{1,5,6,7\}=\{1,6\}$.  The case $I\cap \{1,5,6,7\}=\{5,7\}$ is similar.

  If  some $B_{n}(A)$ contains  a vectorial sunflower of $K$ petals of this type  (denote by  $l_1,l_6$  the common first and 5 coordinate and assume that $l_1>l_6$)  we have that there exists an $K$-d.s.v. $\overline x_j=(x_{5j},x_{7j}),\ \text{Set}(\overline x_j) \subset A,\ j=1,\dots K$ such that $x_{7j}-x_{5j}=l_1-l_6$. Thus, for $r=l_1-l_6$ we have that $V_{2r}(A)$ contains a $K$-d.s.v. and Lemma \ref{UUU} implies the Claim for vectorial sunflowers of this type.
\end{itemize}
 \end{itemize}
\end{proof}
\section{Expected values} We define the quantities:
\begin{eqnarray*}
\sigma_{\alpha,\beta}(n)&=&\sum_{\substack{x,y\ge 1\\ x+y=n}}x^{-\alpha}y^{-\beta}= \sum_{\substack{1\le x<n}}x^{-\alpha}(n-x)^{-\beta},\\
\tau_{\alpha,\beta}(n)&=&\sum_{\substack{x,y\ge 1\\ x-y=n}}x^{-\alpha}y^{-\beta}=\sum_{\substack{1\le x}}x^{-\alpha}(n+x)^{-\beta}
\end{eqnarray*}
and in general
\begin{eqnarray*}
\sigma_{\alpha,\beta}(n;m)=\sum_{\substack{x,y>m\\ x+y=n}}x^{-\alpha}y^{-\beta},\qquad
\tau_{\alpha,\beta}(n;m)=\sum_{\substack{x,y>m\\ x-y=n}}x^{-\alpha}y^{-\beta}.
\end{eqnarray*}
The next Lemma will be apply later many times. We will write $\stackrel{*}{\ll} $ to mean that we are using Lemma \ref{ab} in an inequality. All $x_i$ appearing in this section are positive integers.
\begin{lem}\label{ab}For any $\alpha,\beta<1$ with $\alpha+\beta>1$ we have
\begin{itemize}
  \item [i)] $\sigma_{\alpha,\beta}(n;m)\ll  (n+m)^{1-\alpha-\beta}.\qquad \quad {\rm iii)}\ \sigma_{\alpha,\beta}(n)\ll  n^{1-\alpha-\beta}.$
  \item [ii)] $ \tau_{\alpha,\beta}(n;m)\ll  (n+m)^{1-\alpha-\beta}.\qquad \quad \ {\rm iv)}\ \tau_{\alpha,\beta}(n)\ll  n^{1-\alpha-\beta}.$
\end{itemize}
\end{lem}
\begin{proof}  If $n<2m$, i) holds because $\sigma_{\alpha,\beta}(n;m)=0$. If $n\ge 2m$ we have
\begin{eqnarray*}
\sigma_{\alpha,\beta}(n;m)&\le  & \sum_{1\le x\le n/2}x^{-\alpha}(n-x)^{-\beta}+\sum_{n/2<x< n}x^{-\alpha}(n-x)^{-\beta}\\
&\ll & \sum_{1\le x\le n/2}x^{-\alpha}n^{-\beta}+\sum_{n/2<x<n}n^{-\alpha}(n-x)^{-\beta}\\
&\ll & n^{1-\alpha-\beta}\ll (n+m)^{1-\alpha-\beta}.
\end{eqnarray*}
To prove ii) we distinguish two cases.  If $n<m$  we have
\begin{eqnarray*}
\tau_{\alpha,\beta}(n;m)\le \sum_{x> m}x^{-\alpha}(n+x)^{-\beta}&\le  & \sum_{x> m}x^{-\alpha-\beta}\ll m^{1-\alpha-\beta}\ll (n+m)^{1-\alpha-\beta}.
\end{eqnarray*}
If $n\ge m$ we have
\begin{eqnarray*}
\tau_{\alpha,\beta}(n;m)\le \sum_{x> m}x^{-\alpha}(n+x)^{-\beta}&= & \sum_{m< x<n}x^{-\alpha}(n+x)^{-\beta}+\sum_{x\ge n}x^{-\alpha}(n+x)^{-\beta}\\
&\ll &n^{-\beta}\sum_{1\le x<n}x^{-\alpha}+\sum_{x\ge n}x^{-\alpha-\beta}\ll n^{1-\alpha-\beta}\ll (n+m)^{1-\alpha-\beta}.
\end{eqnarray*}
The cases iii) and iv) follow from i) and ii) taking $m=0$.
\end{proof}
\begin{lem}\label{abab} Let $a,b$ be positive integers. Then for any $\gamma,\quad 1/2<\gamma<1$,
$$\sum_{1\le x}x^{-\gamma}(x+a)^{-\gamma}(x+b)^{1-2\gamma}\ll (ab)^{1-2\gamma}.$$
\end{lem}
\begin{proof}
Suppose that $a<b$ and split the sum:
\begin{eqnarray*}
S&=&\sum_{x\le b}x^{-\gamma}(x+a)^{-\gamma}(x+b)^{1-2\gamma}+
\sum_{x> b}x^{-\gamma}(x+a)^{-\gamma}(x+b)^{1-2\gamma}\\
&\ll &b^{1-2\gamma}\sum_{x\le b}x^{-\gamma}(x+a)^{-\gamma}+\sum_{x> b}x^{1-4\gamma}{\stackrel{*}{\ll}} b^{1-2\gamma}a^{1-2\gamma}+b^{2-4\gamma}\ll (ab)^{1-2\gamma}.
\end{eqnarray*}
\end{proof}
\subsection{Expected values in $\mathcal S_m(7/11,{\scriptstyle S\pmod N})$}
\begin{lem}\label{VV1}We  have
\begin{enumerate}
\item[i)]$\E(|U_{2r}(A)|)\ll  (r+m)^{-3/11}.$
  \item [ii)]$\E(|V_{2r}(A)|)\ll  (r+m)^{-3/11}.$
  \item [iii)]$\E(|W_{r}(A)|)\ll  (r+m)^{-2/11}.$
\end{enumerate}
\end{lem}
\begin{proof}
\begin{eqnarray*}
\E(|U_{2r}(A)|)&=& \sum_{\substack{x,y>m\\x+y=r}}(xy)^{-\gamma}\ll (r+m)^{1-2\gamma}\ll (r+m)^{-3/11}.\qquad \qquad \qquad \qquad \qquad \\
\E(|V_{2r}(A)|)&=&  \sum_{\substack{x,y>m\\x-y=m}}(xy)^{-\gamma}\ll (r+m)^{1-2\gamma}\ll (r+m)^{-3/11}.\qquad \qquad \qquad \qquad \qquad  \\
\E(|W_{r}(A)|)&\le &\sum_{\substack{x_4,x_5,x_6,x_7,x_8>m\\x_5+x_6=x_7+x_8=r+x_4}}(x_4x_5x_6x_7x_8)^{-\gamma}
\ll \sum_{x_4}x_4^{-\gamma}\Big (\sum_{\substack{x,y>m\\ x+y=r+x_4}}(xy)^{-\gamma}\Big )^2\\
&{\stackrel{*}{\ll}} & \sum_{x_4}x_4^{-\gamma}(r+m+x_4)^{2-4\gamma}{\stackrel{*}{\ll}} (r+m)^{3-5\gamma}\ll (r+m)^{-2/11}.
\end{eqnarray*}
\end{proof}

\begin{lem}\label{Qn}
$\E(|Q_n(A)|)\gg n^{1/11}$ for $n$ large enough.
\end{lem}
\begin{proof} We have
$\E(|Q_{n}(A)|)=\sum_{\{x_1,x_2,x_3\}\in Q_n}\Po(x_1,x_2,x_3\in A)\ge n^{-3\gamma}|Q'_n|,$ where $$Q_n'=\Big \{ \{x_1,x_2,x_3\}\in Q_n:\ x_i\equiv S\pmod N,\ x_i>m\Big \}.$$
We observe that $S\subset \Z_N$ is such that $n\equiv s_1+s_2+s_3\pmod N$ for some pairwise distinct $s_1,s_2,s_3$. We fix $s_1,s_2,s_3$ and write $x_i=s_i+Ny_i$ and $l=\frac{n-s_1-s_2-s_3}{N}$.
Then $|Q'_n|\ge |Q_n^*| $ where
 \begin{eqnarray*}|Q^*_n|=\left | \Big \{\{y_1,y_2,y_3\}:\ y_1+y_2+y_3=l:\ y_i>m \Big \}\right |\asymp l^2\gg n^{2}, \end{eqnarray*}
 if $l>10mN$. Thus, $\E(|Q_n(A)|)\ge n^{-3\gamma}|Q_n^*|\gg n^{-3\gamma+2}\gg n^{1/11}$ for $n$ large enough.
\end{proof}
\begin{prop}\label{DeltaQn}
$\Delta(Q_n)\ll n^{-2/11}.$
\end{prop}
\begin{proof}
If $\omega\sim \omega'$ with $\omega,\omega'\in Q_n$, both sets have exactly one common element, say $x_1$. Thus
\begin{eqnarray*}\Delta(Q_n)=\sum_{\substack{\omega,\omega'\in Q_n\\ \omega\sim \omega'}}\Po(\omega,\omega'\subset A)&\ll &\sum_{\substack{x_1,x_2,x_3,x_2',x_3'\\ x_2+x_3=n-x_1\\x_2'+x_3'=n-x_1}}(x_1x_2x_3x_2'x_3')^{-\gamma}\\
&\le &\sum_{x_1<n}x_1^{-\gamma}\Big (\sum_{\substack{x,y\\ x+y=n-x_1}}(xy)^{-\gamma}   \Big )^2{\stackrel{*}{\ll}} \sum_{x_1<n}x_1^{-\gamma} (n-x_1)^{2-4\gamma}\\ &{\stackrel{*}{\ll}} & n^{3-5\gamma}\ll n^{-2/11}. \end{eqnarray*}
\end{proof}
\begin{lem}\label{Tn}
$\E(|T_n(A)|)\ll (n+m)^{-1/11}. $
\end{lem}
\begin{proof} It is clear that $\E(|T_n(A)|)=0$ if $n<3m$, so it is enough to prove that $\E(|T_n(A)|)\ll n^{-1/11}. $

We observe that if $(x_1,\dots,x_8)\in T_n$ then some of these restrictions happens:
\begin{itemize}
  \item [i)] All $x_i$ are pairwise distinct.
  \item [ii)] $x_7=x_8$ and $x_1,x_2,x_3,x_4,x_5,x_6,x_7$ are pairwise distinct.
  \item [iii)] $x_4\in \{x_2,x_3\}$ and $x_1,x_2,x_3,x_5,x_6,x_7,x_8$ are pairwise distinct.
  \item [iv)] $x_6\in \{x_2,x_3\}$ and $x_1,x_2,x_3,x_4,x_5,x_7,x_8$ are pairwise distinct.
  \item [v)] $x_8\in \{x_2,x_3\}$ and $x_1,x_2,x_3,x_4,x_5,x_6,x_7$ are pairwise distinct.
\end{itemize}
In order to simplify these conditions we observe that iv) and v) are essentially the
same condition and that $x_2$ and $x_3$ play the same role, so iii) can be substituted by
$x_4 = x_2$ and iv) and v) by $x_6 = x_2$. Thus we have
\begin{eqnarray*}
\E(|T_n(A)|)&\ll & \sum_{\substack{x_1,x_2,x_3,x_4,x_5,x_6,x_7,x_8\\x_1+x_2+x_3=n\\x_1+x_4=x_5+x_6=x_7+x_8}}(x_1x_2x_3x_4x_5x_6x_7x_8)^{-\gamma}+\sum_{\substack{x_1,x_2,x_3,x_4,x_5,x_6,x_7
\\x_1+x_2+x_3=n\\x_1+x_4=x_5+x_6=2x_7}}(x_1x_2x_3x_4x_5x_6x_7)^{-\gamma}\\
&+ & \sum_{\substack{x_1,x_2,x_3,x_5,x_6,x_7,x_8\\x_1+x_2+x_3=n\\x_1+x_2=x_5+x_6=x_7+x_8}}(x_1x_2x_3x_5x_6x_7x_8)^{-\gamma}+\sum_{\substack{x_1,x_2,x_3,x_4,x_5,x_7,x_8
\\x_1+x_2+x_3=n\\x_1+x_4=x_5+x_2=x_7+x_8}}(x_1x_2x_3x_4x_5x_7x_8)^{-\gamma}\\
&=&S_1+S_2+S_3+S_4.
\end{eqnarray*}

\begin{eqnarray*}
S_1&\ll &\sum_{x_1,x_4}(x_1x_4)^{-\gamma}\sum_{\substack{x_2,x_3\\ x_2+x_3=n-x_1}}(x_2x_3)^{-\gamma}\sum_{\substack{x_5,x_6\\ x_5+x_6=x_1+x_4}}(x_5x_6)^{-\gamma}
\sum_{\substack{x_7,x_8\\ x_7+x_8=x_1+x_4}}(x_7x_8)^{-\gamma}\\&{\stackrel{*}{\ll}}&  \sum_{x_1,x_4}(x_1x_4)^{-\gamma}(n-x_1)^{1-2\gamma}(x_1+x_4)^{2-4\gamma} {\stackrel{*}{\ll}} \sum_{x_1}x_1^{3-6\gamma}(n-x_1)^{1-2\gamma}{\stackrel{*}{\ll}} n^{5-8\gamma}.\end{eqnarray*}

\begin{eqnarray*}S_2&\ll & \sum_{x_1,x_4}\left (x_1x_4\tfrac{x_1+x_4}2\right )^{-\gamma}\sum_{\substack{x_2,x_3\\x_2+x_3=n-x_1}}(x_2x_3)^{-\gamma}\sum_{\substack{x_5,x_6\\x_5+x_6=x_1+x_4}}(x_5x_6)^{-\gamma}\\
&{\stackrel{*}{\ll}}& \sum_{x_1,x_4}\left (x_1x_4(x_1+x_4)\right )^{-\gamma}(n-x_1)^{1-2\gamma}(x_1+x_4)^{1-2\gamma} \\
&\ll & \sum_{x_1}x_1^{-\gamma}(n-x_1)^{1-2\gamma}\sum_{x_4}x_4^{-\gamma}(x_1+x_4)^{1-3\gamma}
 {\stackrel{*}{\ll}}   \sum_{x_1}x_1^{2-5\gamma}(n-x_1)^{1-2\gamma}{\stackrel{*}{\ll}}   n^{4-7\gamma}.\end{eqnarray*}

\begin{eqnarray*}S_3&\ll & \sum_{x_1,x_2}(x_1x_2(n-x_1-x_2))^{-\gamma}\sum_{\substack{x_5,x_6>m\\x_5+x_6=x_1+x_2}}(x_4x_5)^{-\gamma}\sum_{\substack{x_7,x_8>m\\x_7+x_8=x_1+x_2}}
(x_7x_8)^{-\gamma}\qquad \\&\ll &\sum_{l}\sum_{\substack{x_1,x_2\\x_1+x_2=l}}(x_1x_2)^{-\gamma}(n-l)^{-\gamma} \sum_{\substack{x_5,x_6>m\\x_5+x_6=l}}(x_4x_5)^{-\gamma}\sum_{\substack{x_7,x_8>m\\x_7+x_8=l}}
(x_7x_8)^{-\gamma}\\ &{\stackrel{*}{\ll}}& \sum_{l}(n-l)^{-\gamma}l^{3-6\gamma}{\stackrel{*}{\ll}} n^{4-7\gamma}.\qquad \end{eqnarray*}
The estimate of $S_4$ is more involved.  We observe  that given $x_1,x_3,x_4$ the values of $x_2$ and $x_5$ are determined by
$$x_2=n-x_3-x_1,\qquad x_5=x_4+2x_1+x_3-n.$$
\begin{eqnarray*}S_4&\ll & \sum_{x_1}\sum_{x_3<n-x_1}\sum_{x_4}x_1^{-\gamma}(n-x_3-x_1)^{-\gamma}x_3^{-\gamma}x_4^{-\gamma}(x_4+2x_1+x_3-n)^{-\gamma}\sum_{\substack{x_7,x_8\\x_7+x_8=x_1+x_4}}(x_7x_8)^{-\gamma}
\\
&{\stackrel{*}{\ll}} & \sum_{x_1}\sum_{x_3<n-x_1}x_3^{-\gamma}x_1^{-\gamma}(n-x_3-x_1)^{-\gamma}\sum_{x_4}x_4^{-\gamma}(x_4+2x_1+x_3-n)^{-\gamma}(x_4+x_1)^{1-2\gamma}.\end{eqnarray*}
Now we apply Lemma \ref{abab} to the last sum and later we write $x_3=n-x_1-z,\ z>0$ to get
\begin{eqnarray*}
S_4&\ll & \sum_{x_1}\sum_{x_3<n-x_1}x_3^{-\gamma}x_1^{1-3\gamma}(n-x_3-x_1)^{-\gamma}(2x_1+x_3-n)^{1-2\gamma}\\
&\ll & \sum_{x_1}\sum_{z<n-x_1}(n-x_1-z)^{-\gamma}x_1^{1-3\gamma}z^{-\gamma}(x_1+z)^{1-2\gamma}\\
&\ll & \sum_{x_1}x_1^{2-5\gamma}\sum_{z<n-x_1}(n-x_1-z)^{-\gamma}z^{-\gamma}
{\stackrel{*}{\ll}} \sum_{x_1}x_1^{2-5\gamma}(n-x_1)^{1-2\gamma}{\stackrel{*}{\ll}} n^{4-7\gamma}.\qquad \qquad
\end{eqnarray*}
\end{proof}

\subsection{Expected values in $\mathcal S_m(\frac 23+\frac{\varepsilon}{9+9\varepsilon},{\scriptstyle S\pmod N})$}
\begin{lem}\label{VV2}We have
\begin{enumerate}
\item[i)]$\E(|U_{2r}(A)|)\ll (r+m)^{-1/3},\qquad \quad \text{\rm ii) } \E(|U_{3r}(A)|)\ll (r+m)^{-\varepsilon/6}.$
  \item [iii)]$\E(|V_{2r}(A)|)\ll (r+m)^{-1/3},\qquad \quad \text{\rm iv) } \E(|V_{3r}(A)|)\ll (r+m)^{-\varepsilon/6}.$
\end{enumerate}
\end{lem}
\begin{proof}
\begin{eqnarray*}
\E(|U_{2r}(A)|)&=& \sum_{\substack{x,y>m\\x+y=r}}(xy)^{-\gamma}\ll (r+m)^{1-2\gamma}\ll (r+m)^{-1/3}.\qquad \qquad \qquad \qquad \qquad \\
\E(|V_{2r}(A)|)&=&  \sum_{\substack{x,y>m\\x-y=m}}(xy)^{-\gamma}\ll (r+m)^{1-2\gamma}\ll (r+m)^{-1/3}.\qquad \qquad \qquad \qquad \qquad \end{eqnarray*}\begin{eqnarray*}
\E(|U_{3r}(A)|)&\le & \sum_{\substack{x,y,z>m\\x+y+z=r}}(xyz)^{-\gamma}\le \sum_{z>0}z^{-\gamma}\sum_{\substack{x,y>m\\x+y=r-z}}(xy)^{-\gamma}\\ &{\stackrel{*}{\ll}} & \sum_{z}z^{-\gamma}(r-z+m)^{1-2\gamma}{\stackrel{*}{\ll}} (r+m)^{2-3\gamma}\ll (r+m)^{-\varepsilon/6}.\qquad \qquad \qquad \\
\E(|V_{3r}(A)|)&\le & \sum_{\substack{x,y,z>m\\x+y-z=r}}(xyz)^{-\gamma}=\sum_{z}z^{-\gamma}\sum_{\substack{x,y>m\\x+y=r+z}}(xy)^{-\gamma}\\ &{\stackrel{*}{\ll}} & \sum_{z}z^{-\gamma}(r+z+m)^{1-2\gamma}{\stackrel{*}{\ll}} (r+m)^{2-3\gamma}\ll (r+m)^{-\varepsilon/6}.\qquad \qquad \qquad  \end{eqnarray*}
\end{proof}
\begin{lem}\label{Rn}
$\E(R_{n}(A))\gg n^{\frac{2\varepsilon^2}{9+9\varepsilon}}.$
\end{lem}
\begin{proof}
We have $\E(|R_{n}(A)|)=\sum_{\{x_1,x_2,x_3,x_4\}\in R_n}\Po(x_1,x_2,x_3,x_4\in A)\ge n^{-(3+\varepsilon)\gamma}|R'_n|,$ where
$$R_n'=\{\{x_1,x_2,x_3,x_4\}\in R_n:\ x_i\equiv S\pmod N,\ x_i>m\}.$$
We observe that $S\subset \Z_{N}$ is such that $n\equiv s_1+s_2+s_3+s_4\pmod{N}$ for some pairwise distinct $s_1,s_2,s_3,s_4$. We fix $s_1,s_2,s_3,s_4$ and write $x_i=s_i+256y_i$ and $l=\frac{n-s_1-s_2-s_3-s_4}{N}$.
Then $|R'_n|\ge |R_n^*| $ where
 \begin{eqnarray*}|R^*_n|&=&\#\left \{\{y_1,y_2,y_3,y_4\}:\ y_1+y_2+y_3+y_4=l,\ m<\min(y_1,y_2,y_3,y_4)\le n^{\varepsilon}/256 \right \}\\ &\asymp & n^{\varepsilon}l^2\asymp n^{2+\varepsilon} \end{eqnarray*}
Thus, $\E(|R_n(A)|)\ge n^{-(3+\varepsilon)\gamma}|R_n^*|\gg n^{-(3+\varepsilon)\gamma+2+\varepsilon}\gg n^{\frac{2\varepsilon^2}{9+9\varepsilon}}$.
\end{proof}
\begin{prop}\label{Deltan}
$\Delta(R_n)\ll n^{\frac{-3\varepsilon+2\varepsilon^2}{9+9\varepsilon}}. $
\end{prop}
\begin{proof} We have that $$\Delta(R_n)=\sum_{\substack{\omega,\omega'\in R_n\\ \omega\sim\omega'}}\Po(\omega,\omega'\in A)$$
We split $\Delta(R_n) $ in several sums according to the common elements of $\omega$ and $\omega'$. We suppose that $i=1$ is the index for which $x_1\le n^{\varepsilon}$.

  1. $\omega\cap \omega'=x_1$ \begin{eqnarray*}\sum_{\substack{x_1,x_2,x_3,x_4,x_2',x'_3,x'_4\\ x_1+x_2+x_3+x_4=n\\x_1+x'_2+x'_3+x'_4=n\\x_1\le n^{\varepsilon}}}(x_1x_2x_3x_4x_2'x_3'x_4')^{-\gamma}&\ll& \sum_{x_1\le n^{\varepsilon}}x_1^{-\gamma}\Big (\sum_{\substack{x_2,x_3,x_4\\ x_2+x_3+x_4=n-x_1}}(x_2x_3x_4)^{-\gamma}\Big )^2\\ & {\stackrel{*}{\ll}} & \sum_{x_1\le n^{\varepsilon}}x_1^{-\gamma} (n-x_1)^{4-6\gamma}\ll n^{4-6\gamma+\varepsilon(1-\gamma)}. \end{eqnarray*}
  2. $\omega\cap \omega'=x_j$ for some $j=2,3,4$. Without lost of generality we consider the case $x_2=x_2'$: \begin{eqnarray*} \sum_{\substack{x_1,x_2,x_3,x_4,x_1',x'_3,x'_4\\x_1,x_1'\le n^{\varepsilon}\\x_3+x_4=n-x_2-x_1\\x_3'+x_4'=n-x_2-x_1'\\}}(x_1x_2x_3x_4x_1'x_3'x_4')^{-\gamma}&\ll & \sum_{x_2}x_2^{-\gamma}\Big (\sum_{x_1\le n^{\varepsilon}}x_1^{-\gamma}\sum_{\substack{x,y\\x+y=n-x_2-x_1}}(xy)^{-\gamma}\Big )^2\end{eqnarray*}
  Now we split the sum in two sums according to $x_2\le n-2n^{\varepsilon}$ or $n-2n^{\varepsilon}<x_2< n$.
\begin{eqnarray*}
     \sum_{x_2\le n-2n^{\varepsilon}}x_2^{-\gamma}\left (\sum_{x_1\le n^{\varepsilon}}x_1^{-\gamma}(n-x_2-x_1)^{1-2\gamma}\right )^2&\ll &\sum_{x_2\le n}x_2^{-\gamma}\left (\frac{n-x_2}2\right )^{2-4\gamma}\left (\sum_{x_1\le n^{\varepsilon}}x_1^{-\gamma}\right )^2\\
     &{\stackrel{*}{\ll}} &n^{3-5\gamma}n^{2(1-\gamma)\varepsilon}\ll n^{(2-2\gamma)\varepsilon+3-5\gamma}.
       \end{eqnarray*}
       \begin{eqnarray*}
       \sum_{n-2n^{\varepsilon}<x_2\le n}x_2^{-\gamma}\left (\sum_{x_1\le n^{\varepsilon}}x_1^{-\gamma}(n-x_2-x_1)^{1-2\gamma}\right )^2&{\stackrel{*}{\ll}} &\sum_{n-2n^{\varepsilon}<x_2\le n}x_2^{-\gamma}(n-x_2)^{4-6\gamma}\\ &\ll &n^{-\gamma}\sum_{n-2n^{\varepsilon}<x_2\le n}(n-x_2)^{4-6\gamma}\ll n^{-\gamma+(5-6\gamma)\varepsilon}.
       \end{eqnarray*}
  3. $\omega\cap \omega'=\{x_1,x_j\}$ for some $j=2,3,4$. Without lost of generality we consider the case $x_1=x_1'$ and $x_2=x_2'$:
 \begin{eqnarray*}\sum_{\substack{x_1,x_2,x_3,x_4,x'_3,x'_4\\ x_1+x_2+x_3+x_4=n\\x_1+x_2+x_3'+x_4'=n\\ x_1\le n^{\varepsilon}}}(x_1x_2x_3x_4x_3'x_4')^{-\gamma}&\ll &\sum_{x_1\le n^{\varepsilon}}x_1^{-\gamma}\sum_{x_2}x_2^{-\gamma}
 \Big (\sum_{\substack{x_3,x_4\\x_3+x_4=n-x_1-x_2}}(x_3x_4)^{-\gamma}  \Big )^2\\
 &{\stackrel{*}{\ll}}&  \sum_{x_1\le n^{\varepsilon}}x_1^{-\gamma}\sum_{x_2}x_2^{-\gamma}(n-x_1-x_2)^{2-4\gamma}  \\&{\stackrel{*}{\ll}} & \sum_{x_1\le n^{\varepsilon}}x_1^{-\gamma}(n-x_1)^{3-5\gamma}\ll n^{3-5\gamma+\varepsilon(1-\gamma)}.
 \end{eqnarray*}
  4. $\omega\cap\omega'=\{x_j,x_k\}$ for some $2\le j<k\le 4$. Without lost of generality we consider the case $x_2=x_2'$ and $x_3=x_3'$:
  \begin{eqnarray*}& &\sum_{\substack{x_1,x_2,x_3,x_4,x'_1,x'_4\\ x_1+x_2+x_3+x_4=n\\x_1'+x_2+x_3+x_4'=n\\ x_1,x_1'\le n^{\varepsilon}}}(x_1x_2x_3x_4x_1'x_4')^{-\gamma}\ll \sum_{\substack{x_2,x_3\\x_2+x_3<n}}(x_2x_3)^{-\gamma}\left (\sum_{x_1\le n^{\varepsilon}}x_1^{-\gamma}(n-x_2-x_3-x_1)^{-\gamma}\right )^2\\
  &{\stackrel{*}{\ll}} & \sum_{\substack{x_2,x_3\\x_2+x_3<n-2n^{\varepsilon}}}(x_2x_3)^{-\gamma}\left (n^{-\varepsilon \gamma}\sum_{x_1\le n^{\varepsilon}}x_1^{-\gamma} \right )^2+\sum_{\substack{x_2,x_3\\n-2n^{\varepsilon}\le x_2+x_3<n}}(x_2x_3)^{-\gamma}\left ((n-x_2-x_3)^{1-2\gamma}\right )^2\\ &\ll & \sum_{\substack{x_2,x_3\\x_2+x_3<n-2n^{\varepsilon}}}(x_2x_3)^{-\gamma}\left (n^{-\varepsilon \gamma}n^{\varepsilon(1-\gamma)}\right )^2  +\sum_{n-2n^{\varepsilon}\le l<n }  \sum_{\substack{x_2,x_3\\ x_2+x_3=l}}(x_2x_3)^{-\gamma}(n-l)^{2-4\gamma}\\
  &{\stackrel{*}{\ll}} &\sum_{\substack{x_2,x_3\\x_2+x_3<n-2n^{\varepsilon}}}(x_2x_3)^{-\gamma}n^{\varepsilon(2-4\gamma)}  +n^{1-2\gamma}\sum_{n-2n^{\varepsilon}\le l<n }(n-l)^{2-4\gamma}
  {\stackrel{*}{\ll}} n^{1-2\gamma+\varepsilon(2-4\gamma)}+ n^{1-2\gamma+\varepsilon(3-4\gamma)}.
  \end{eqnarray*}

Observe that if $\omega \ne \omega'$ it is not possible that they have  three common coordinates.
Putting $\gamma=\frac 23+\frac{\varepsilon}{9+9\varepsilon}$ in each estimate we have that \begin{eqnarray*}\Delta(R_n)&\ll & n^{4-6\gamma+\varepsilon(1-\gamma)}+n^{3-5\gamma+ \varepsilon(1-\gamma)} + n^{-\gamma+\varepsilon(5-6\gamma)}+n^{1-2\gamma+\varepsilon(2-4\gamma)}+ n^{1-2\gamma+\varepsilon(3-4\gamma)}\\ &\ll & n^{\frac{-3\varepsilon+2\varepsilon^2}{9+9\varepsilon}}+ n^{\frac{-3-5\varepsilon+2\varepsilon^2}{9+9\varepsilon}}+n^{\frac{-6+2\varepsilon+3\varepsilon^2}{9+9\varepsilon}}
+n^{\frac{-3-11\varepsilon-10\varepsilon^2}{9+9\varepsilon}}+ n^{\frac{-3-2\varepsilon-\varepsilon^2}{9+9\varepsilon}} \ll n^{\frac{-3\varepsilon+2\varepsilon^2}{9+9\varepsilon}}.\end{eqnarray*}
\end{proof}

\begin{lem}\label{B1n}
$\E(B_{n}(A))\ll (n+m)^{-\frac{\varepsilon^2}{18}}.$
\end{lem}
\begin{proof} It is clear that $\E(|B_n(A)|)=0$ if $n<4m$, so it is enough to prove that $\E(B_{n}(A))\ll n^{-\frac{\varepsilon^2}{9+9\varepsilon}}.$

We observe that if $(x_1,\dots, x_7)\in B_{n}(A)$   then some of these restrictions holds:
\begin{itemize}
  \item [i)]All $x_i$ are pairwise distinct.
  \item [ii)]$x_6=x_7$ and all $x_1,x_2,x_3,x_4,x_5,x_6$ are pairwise distinct.
  \item [iii)]$x_5\in \{x_2,x_3,x_4\}$ and all $x_1,x_2,x_3,x_4,x_6,x_7$ are pairwise distinct.
  \item [iv)]$x_6\in \{x_2,x_3,x_4\}$ and all $x_1,x_2,x_3,x_4,x_5,x_7$ are pairwise distinct.
\item [v)] $x_7\in \{x_2,x_3,x_4,\}$ and all $x_1,x_2,x_3,x_4,x_5,x_6$ are pairwise distinct.
\end{itemize}Thus we have
$$ \E(|B_{n}(A)|)\le  \sum'_{\substack{(x_1,x_2,x_3,x_4,x_5,x_6, x_7)\\ x_1+x_2+x_3 +x_4=n\\ x_1+x_5=x_6+x_7\\\min(x_1,x_2,x_3,x_4)\le n^{\varepsilon}}}\Po(x_1,\dots,x_7\in A)$$ and $\sum'$ means that $\overline x$ satisfies  i), ii), iii), iv) or v).

In order to simplify these conditions we observe that iv) and v) are essentially the same condition and that $x_2,x_3,x_4$ play the same role, so iii) can be substituted by $x_5=x_2$ and iv) and v) by $x_7=x_2$.  Thus we have that
\begin{eqnarray*}
\E(B_{n}(A))&\ll &\sum_{\substack{x_1,x_2,x_3,x_4,x_5,x_6, x_7\\ x_1+x_2+x_3 +x_4=n\\ x_1+x_5=x_6+x_7\\\min(x_1,x_2,x_3,x_4)\le n^{\varepsilon}}}(x_1x_2x_3x_4x_5x_6x_7)^{-\gamma}+
\sum_{\substack{x_1,x_2,x_3,x_4,x_5,x_6\\ x_1+x_2+x_3 +x_4=n\\ x_1+x_5=2x_6\\\min(x_1,x_2,x_3,x_4)\le n^{\varepsilon}}}(x_1x_2x_3x_4x_5x_6)^{-\gamma}\\
&+&
\sum_{\substack{x_1,x_2,x_3,x_4,x_5,x_6\\ x_1+x_2+x_3 +x_4=n\\ x_1+x_2=x_6+x_7\\\min(x_1,x_2,x_3,x_4)\le n^{\varepsilon}}}(x_1x_2x_3x_4x_6x_7)^{-\gamma}+
\sum_{\substack{x_1,x_2,x_3,x_4,x_5,x_6\\ x_1+x_2+x_3 +x_4=n\\ x_1+x_5=x_6+x_2\\\min(x_1,x_2,x_3,x_4)\le n^{\varepsilon}}}(x_1x_2x_3x_4x_5x_6)^{-\gamma}\\ &=& S_1+S_2+S_3+S_4.
\end{eqnarray*}
We write $S_1\le S_1'+S_1''$  to distinguish the cases $x_1\le n^{\varepsilon}$ and $x_2\le n^{\varepsilon}$.
\begin{eqnarray*}S_1'&\ll &\sum_{x_1\le n^{\varepsilon}}\sum_{x_2<n}\sum_{x_5}(x_1x_2x_5)^{-\gamma} \sum_{\substack{x_3,x_4\\x_3+x_4=n-x_1-x_2}}(x_3x_4)^{-\gamma}\sum_{\substack{x_6,x_7\\x_6+x_7=x_1+x_5}}(x_6x_7)^{-\gamma}\\
&{\stackrel{*}{\ll}} & \sum_{x_1\le n^{\varepsilon}}\sum_{x_2<n}\sum_{x_5}(x_1x_2x_5)^{-\gamma} (n-x_1-x_2)^{1-2\gamma}(x_1+x_5)^{1-2\gamma}\\
&{\stackrel{*}{\ll}} & \sum_{x_1\le n^{\varepsilon}}\sum_{x_2<n}x_1^{2-4\gamma}x_2^{-\gamma} (n-x_1-x_2)^{1-2\gamma}\ll  \sum_{x_1\le n^{\varepsilon}}x_1^{2-4\gamma} (n-x_1)^{2-3\gamma}\\&{\stackrel{*}{\ll}}& n^{2-3\gamma+\varepsilon(3-4\gamma)}\ll n^{-\frac{\varepsilon}{3+3\varepsilon}+\varepsilon(\frac 13-\frac{4\varepsilon}{9+9\varepsilon})}\ll n^{-\frac{\varepsilon^2}{9+9\varepsilon}}\ll n^{-\varepsilon^2/18}.
 \end{eqnarray*}
 \begin{eqnarray*}S_1''&\ll &\sum_{x_2\le n^{\varepsilon}}\sum_{x_1<n}\sum_{x_5}(x_1x_2x_5)^{-\gamma} \sum_{\substack{x_3,x_4\\x_3+x_4=n-x_1-x_2}}(x_3x_4)^{-\gamma}\sum_{\substack{x_6,x_7\\x_6+x_7=x_1+x_5}}(x_6x_7)^{-\gamma}\\
&{\stackrel{*}{\ll}} & \sum_{x_2\le n^{\varepsilon}}\sum_{x_1<n}\sum_{x_5}(x_1x_2x_5)^{-\gamma} (n-x_1-x_2)^{1-2\gamma}(x_1+x_5)^{1-2\gamma}\\
&{\stackrel{*}{\ll}} & \sum_{x_2\le n^{\varepsilon}}\sum_{x_1<n}x_1^{2-4\gamma}x_2^{-\gamma} (n-x_1-x_2)^{1-2\gamma}{\stackrel{*}{\ll}} \sum_{x_2\le n^{\varepsilon}}x_2^{-\gamma} (n-x_2)^{4-6\gamma}\\&{\stackrel{*}{\ll}}& n^{4-6\gamma+\varepsilon(1-\gamma)}\ll n^{-\frac{2\varepsilon}{3+3\varepsilon}+\varepsilon(\frac 13-\frac{\varepsilon}{9+9\varepsilon}  )}\ll n^{\frac{-3\varepsilon+2\varepsilon^2}{9+9\varepsilon}}\ll n^{-\varepsilon^2/18}.
 \end{eqnarray*}
In the estimates of $S_2,S_3$ and $S_4$ we remove the annoying condition $\min(x_1,x_2,x_3,x_4)\le n^{\varepsilon}$.
\begin{eqnarray*}
S_2&\le &\sum_{\substack{x_1,x_2,x_3,x_4,x_5,x_6\\ x_1+x_2+x_3 +x_4=n\\ x_1+x_5=2x_6}}(x_1x_2x_3x_4x_5x_6)^{-\gamma}\\ &\le & \sum_{x_1,x_2}(x_1x_2)^{-\gamma}\Big ( \sum_{x_5}(x_5(x_1+x_5)/2)^{-\gamma}\Big )\Big (\sum_{\substack{x_3,x_4\\ x_3+x_4=n-x_1-x_2}}(x_3x_4)^{-\gamma}\Big )\\
&{\stackrel{*}{\ll}} &\sum_{x_1,x_2}(x_1x_2)^{-\gamma}x_1^{1-2\gamma}(n-x_1-x_2)^{1-2\gamma}{\stackrel{*}{\ll}} \sum_{x_1}x_1^{1-3\gamma}(n-x_1)^{2-3\gamma}\ll n^{4-6\gamma}.
\end{eqnarray*}

\smallskip

\begin{eqnarray*}
S_3&\le &\sum_{\substack{x_1,x_2,x_3,x_4,x_5,x_6\\ x_1+x_2+x_3 +x_4=n\\ x_1+x_2=x_6+x_7}}(x_1x_2x_3x_4x_6x_7)^{-\gamma}\\ &\le & \sum_{x_1,x_2}(x_1x_2)^{-\gamma}\sum_{\substack{x_6,x_7\\x_6+x_7=x_1+x_2}}(x_6x_7)^{-\gamma}  \sum_{\substack{x_3,x_4\\x_3+x_4=n-x_1-x_2}}(x_3x_4)^{-\gamma} \\
&{\stackrel{*}{\ll}} &\sum_{x_1, x_2}(x_1x_2)^{-\gamma}(x_1+x_2)^{1-2\gamma}(n-x_1-x_2)^{1-2\gamma}\\
&\ll& \sum_l \sum_{\substack{x_1,x_2\\x_1+x_2=l}}(x_1x_2)^{-\gamma}l^{1-2\gamma}(n-l)^{1-2\gamma}
{\stackrel{*}{\ll}} \sum_l l^{2-4\gamma}(n-l)^{1-2\gamma}{\stackrel{*}{\ll}} n^{4-6\gamma}.\qquad
\end{eqnarray*}

\smallskip

\begin{eqnarray*}
S_4&\le &\sum_{\substack{x_1,x_2,x_3,x_4,x_5,x_6\\ x_1+x_2+x_3 +x_4=n\\ x_1+x_5=x_6+x_2}}(x_1x_2x_3x_4x_6x_7)^{-\gamma}\\ &\ll & \sum_{\substack{x_1,x_2\\x_1<x_2\\x_1+x_2<n}}(x_1x_2)^{-\gamma}\sum_{\substack{x_5,x_6\\x_5-x_6=x_2-x_1}}(x_6x_5)^{-\gamma}  \sum_{\substack{x_3,x_4\\x_3+x_4=n-x_1-x_2}}(x_3x_4)^{-\gamma} \qquad \qquad \qquad \\
&{\stackrel{*}{\ll}} &\sum_{\substack{x_1, x_2\\x_1<x_2\\x_1+x_2<n}}x_1^{-2\gamma}(x_2-x_1)^{1-2\gamma}(n-x_1-x_2)^{1-2\gamma}\\
&\ll &\sum_{x_1<n/2}x_1^{-2\gamma}(n-2x_1)^{3-4\gamma}\ll\sum_{x_1<n/2}x_1^{-2\gamma}(n/2-x_1)^{3-4\gamma} {\stackrel{*}{\ll}} n^{3-4\gamma}n^{1-2\gamma}.
\end{eqnarray*}
Thus, $S_2,S_3,S_4\ll n^{4-6\gamma}\ll n^{-\frac{2\varepsilon}{3+3\varepsilon}}\ll n^{\frac{-\varepsilon^2}{18}}$.
\end{proof}

\subsection{Acknowledgements} I am grateful to Igor Shparlinski for useful comments about Theorem \ref{GSZ} and to Joel Spencer and Prasad Tetali for provide me a copy of the paper \cite{ST} and some informations about Erd\H os's conjecture. I am also grateful to Rafa Tesoro for a carefully reading of the paper.

\end{document}